\documentclass[12pt]{article}
\usepackage{amsfonts, amsmath, amssymb, latexsym, eucal, amscd} 

\usepackage{amsthm}
\usepackage{amscd}

\usepackage{cite}

\newtheorem{definition}{Definition}[section]
\newtheorem{theorem}[definition]{Theorem}
\newtheorem{lemma}[definition]{Lemma}
\newtheorem{corollary}[definition]{Corollary}
\newtheorem{remark}[definition]{Remark}

\newtheorem{proposition}[definition]{Proposition}

\begin{document} 

\title{\bf Projective geometries,  $Q$-polynomial \\structures,
and quantum groups
}
\author{
Paul Terwilliger}
\date{}
%%\date{\today}

\maketitle
\begin{abstract} In 2023 we obtained a $Q$-polynomial structure for the projective geometry $L_N(q)$. In the present paper,
we display a more general $Q$-polynomial structure for $L_N(q)$. Our new  $Q$-polynomial structure is defined using a free parameter $\varphi$ that takes any positive real value.
For $\varphi=1$ we recover the original $Q$-polynomial structure.
We interpret the new $Q$-polynomial structure using the quantum group $U_{q^{1/2}}(\mathfrak{sl}_2)$ in the equitable presentation.
We use the new $Q$-polynomial structure to obtain analogs of the four split decompositions that appear in the theory
of $Q$-polynomial distance-regular graphs.
\medskip

%%%Secondary 06A11; 05C50.
 \end{abstract}
 
\section{Introduction} Given a finite undirected graph $\Gamma$, 
one often studies the eigenvalues of its adjacency matrix. Of interest is the interaction between these eigenvalues and the combinatorial structure of $\Gamma$;
see for example the 1974 book \cite{biggs}  by N. Biggs. In that book, one of the topics is a type of graph said to be distance-regular \cite[p.~155]{biggs}.
The distance-regular graphs have a certain combinatorial
regularity, that makes the eigenvalues and eigenspace dimensions relatively easy to compute. Two examples of distance-regular graphs are
the Hamming graphs \cite[Section~2.10.2]{bbit}, \cite[Section~9.2]{bcn} and the Johnson graphs \cite[Section~2.10.3]{bbit}, \cite[Section~9.1]{bcn}.
%% these are the distance-regular graphs  introduced by Biggs in 1971 CITEintersectionmat.
\medskip

\noindent
In his 1973 thesis \cite{delsarte} about association schemes and coding theory, P. Delsarte introduced the $Q$-polynomial property for distance-regular graphs.  By Delsarte's definition \cite[Section~5.3]{delsarte}, a $Q$-polynomial distance-regular graph comes with two sequences of orthogonal polynomials that are related by what is now
called Delsarte duality \cite{leonard} or Askey-Wilson duality \cite[p.~261]{bannai}. Delsarte showed that the Hamming and Johnson graphs are $Q$-polynomial, and he described the corresponding orthogonal polynomials \cite[Chapter~4]{delsarte}.
\medskip

\noindent
In their 1985 book \cite[Section~3.6]{bannai}, E.~Bannai and T.~Ito demonstrated  that all the known infinite families of primitive distance-regular graphs with unbounded diameter are $Q$-polynomial.
Based on this evidence, they  conjectured that any primitive distance-regular graph is $Q$-polynomial if the diameter is sufficiently large \cite[p.~312]{bannai}. They also proposed
the classification of the $Q$-polynomial distance-regular graphs \cite[p.~xiii]{bannai}.  To aid in this classification, they gave a detailed version of a result of D.~Leonard \cite{leonard}  that classifies
 the pairs of orthogonal polynomial sequences that satisfy Askey-Wilson duality \cite[Theorem~5.1]{bannai}. For a modern treatment of the Leonard theorem, see \cite[Chapter~6]{bbit} and \cite{LS99,LSnotes}.
\medskip

\noindent  The $Q$-polynomial property has multiple characterizations, such as the cometric condition \cite[Theorem~5.16]{delsarte}, \cite[Definition~11.1]{int}, the balanced set condition \cite[Theorem~1.1]{QPchar}, \cite[Section~19]{int}, and the Pascasio condition \cite[Theorem~1.2]{pascasio}, \cite[Section~17]{int}.
Each of these conditions involves a different point of view.
In the present paper, we adopt the following point of view.
Assume that $\Gamma$ is distance-regular.  Pick a vertex $x$ of $\Gamma$ and consider the subconstituents of $\Gamma$ with respect to $x$ \cite[Section~2]{int}.
By the triangle inequality, the adjacency matrix $A$ of $\Gamma$ acts on these subconstituents in a block-tridiagonal fashion.
We now add the assumption that $\Gamma$ is $Q$-polynomial.
 It was shown in \cite[Section~13]{int} that for each vertex $x$ of $\Gamma$, there exists a diagonal matrix $A^*=A^*(x)$ such that the eigenspaces of $A^*$ are the subconstituents of $\Gamma$ with respect to $x$, and
 $A^*$ acts on the eigenspaces of $A$ in a block-tridiagonal fashion. In summary, each of $A, A^*$ acts on the eigenspaces of the other one in a block-tridiagonal fashion.
The matrix $A^*$ is called the dual adjacency matrix of $\Gamma$
with respect to $x$ \cite[Section~2]{ds}. For information about $A^*$ and related topics, see \cite{bbit, augIto, tSub1,tSub2,tSub3, someAlg, qSerre}.
\medskip
%%%It was shown in CITE that $A, A^*$ satisfy a pair of relations called the tridiagonal relations; these generalize the cubic $q$-Serre relations.

\noindent Motivated by the dual adjacency matrix point of view, in \cite[Section~20]{int} we generalized the $Q$-polynomial property in three directions: (i) we drop the assumption that $\Gamma$ is distance-regular;
(ii) we drop the assumption that every vertex of $\Gamma$ has a dual adjacency matrix, and instead require that one distinguished vertex of $\Gamma$ has a dual adjacency matrix; (iii) we replace
the adjacency matrix by a weighted adjacency matrix.
\medskip

\noindent We comment on the definition of a weighted adjacency matrix. According to \cite[Definition~2.1]{Lnq}, a weighted adjacency matrix is obtained from the adjacency matrix by replacing
each nonzero entry by an arbitrary nonzero scalar. In the present paper, we will relax this definition a bit. In the present paper,  a weighted adjacency matrix is obtained from the adjacency
matrix by (i) replacing each nonzero entry by an arbitrary nonzero scalar, and (ii) replacing each diagonal entry by an arbitrary scalar.
\medskip

\noindent  We now describe some graphs that are not distance-regular, but are $Q$-polynomial in the above generalized sense.
In this description we identify a poset with  its Hasse diagram, which we view as an undirected graph.
In \cite{Lnq} we displayed a $Q$-polynomial structure for the projective geometry $L_N(q)$.
In \cite{Atten} we displayed a $Q$-polynomial structure for the attenuated space poset $\mathcal A_q(N,M)$. In \cite{fullBip} the authors Fern\'{a}ndez, Maleki, Miklavi\v{c}, and Monzillo
obtained a $Q$-polynomial structure for the full bipartite graph of the Hamming graph.
We expect a similar $Q$-polynomial structure for the full bipartite graph of a dual polar graph.
\medskip

\noindent In the present paper, we revisit the projective geometry $L_N(q)$. We have three main results. In our first main result, we display a $Q$-polynomial structure for $L_N(q)$ that is more general than the one in \cite{Lnq}. Our new
$Q$-polynomial structure is defined using a free parameter $\varphi$ that takes any positive real value. For $\varphi=1$ we recover the  $Q$-polynomial structure from \cite{Lnq}. 
In our second main result, we interpret our new $Q$-polynomial structure using the quantum group $U_{q^{1/2}}(\mathfrak{sl}_2)$ in the equitable presentation. 
In our third main result, we use our new $Q$-polynomial structure to obtain analogs of the four split
decompositions that appear in the theory of $Q$-polynomial distance-regular graphs \cite{ds, kim1,kim2,qtet}.
\medskip

\noindent 
We will summarize our main results in more detail, after some brief definitions.
  Let $GF(q)$ denote a finite field with cardinality $q$. Fix an integer $N\geq 1$. Let $\bf V$ denote a vector space over $GF(q)$
that has dimension $N$. Let the set $X$ consist of the subspaces of $\bf V$.  The set $X$, together with the inclusion relation,
is a poset denoted $L_N(q)$.
 We distinguish the zero subspace ${\bf 0} \in X$. We define two matrices $R, L  \in {\rm Mat}_X(\mathbb C)$ with the following entries.
 Let $y,z \in X$.  The matrix $R$ has $(y,z)$-entry $1$ (if $y$ covers $z$), and $0$ (if $y$ does not cover $z$). 
  The matrix $L$ has $(y,z)$-entry $1$ (if $z$ covers $y$), and $0$ (if $z$ does not cover $y$). We call $R$ (resp. $L$) the raising matrix (resp. lowering matrix) for $L_N(q)$.
 \medskip
 
 \noindent We now summarize our first main result.
 We define two matrices $A, A^* \in {\rm Mat}_X(\mathbb C)$ as follows. The matrix $A^*$ is diagonal, with $(y,y)$-entry $q^{-{\rm dim}\,y}$ for all $y\in X$.  Define
 \begin{align*}
A = R + \varphi (A^*)^{-1} L+ \frac{\varphi-1}{q-1}(A^*)^{-1}.
\end{align*}
By construction, $A^*$ has $N+1$  eigenvalues 
\begin{align*}
\theta^*_i = q^{-i} \qquad \qquad (0 \leq i \leq N).
\end{align*}
 By construction, $A$
acts on the eigenspaces of $A^*$ in a block-tridiagonal fashion.
We show that $A$ is diagonalizable, with $N+1$ eigenvalues
\begin{align*}
\theta_i = \frac{\varphi q^{N-i}-q^i}{q-1}  \qquad \qquad (0 \leq i \leq N).
\end{align*}
 We show that $A^*$ acts on the  eigenspaces of $A$ in a block-tridiagonal fashion. Using these facts, we show that $A$ is $Q$-polynomial
  with respect to $\bf 0$, in the sense of Definition  \ref{def:AQpoly} below. 
  \medskip
  
  \noindent We now summarize our second main result. It has to do with the quantum group $U_{q^{1/2}}(\mathfrak{sl}_2)$ in the equitable presentation \cite[Theorem~2.1]{equit}.
In this presentation, the algebra
$U_{q^{1/2}}({\mathfrak{sl}_2})$ is defined by generators
$\mathcal X$, $\mathcal Y$, $\mathcal Y^{-1}$, $\mathcal Z$ and relations $\mathcal Y\mathcal Y^{-1} = \mathcal Y^{-1}\mathcal Y = 1$, 
\begin{align*}
\frac{q\mathcal X \mathcal Y-\mathcal Y \mathcal X}{q-1} = 1,
\qquad \quad
\frac{q  \mathcal Y \mathcal Z- \mathcal Z \mathcal Y}{q-1} = 1,
\qquad \quad
\frac{q \mathcal Z \mathcal X- \mathcal X \mathcal Z}{q-1} = 1. 
\end{align*}
   Define
\begin{align*}
A^+ =  \frac{(A^*)^{-1}}{q-1}-R, \qquad \qquad
A^- =  \frac{(A^*)^{-1}}{q-1} +(A^*)^{-1} L
\end{align*}
and note that
$A =  \varphi A^{-}-A^+$.
We show that
\begin{align*}
&q A^-  A^* - A^* A^- = I, \qquad \qquad  q A^* A^+  -A^+ A^*= I, \\
&   q A^+ A^-   -A^- A^+ = \frac{q^N I}{q-1}. 
\end{align*}
Using these relations, we show that the standard module $V$ for $L_N(q)$ becomes a $U_{q^{1/2}} (\mathfrak{sl}_2)$-module on which the generators $\mathcal X$, $\mathcal Y$, $\mathcal Y^{-1}$, $\mathcal Z$ 
act as follows:
\begin{align*} 
\begin{tabular}[t]{c|cccc}
{\rm generator $g$ }& $\mathcal X$ & $\mathcal Y$ & $\mathcal Y^{-1}$ & $\mathcal Z$
 \\
 \hline
 {\rm action of $g$ on $V$} & $(q-1)q^{-N/2} A^-$ & $q^{N/2}A^*$ & $q^{-N/2} (A^*)^{-1}$ & $(q-1)q^{-N/2} A^+$
    \end{tabular}
\end{align*}
Next we define
\begin{align*}
A_+ =  \frac{(A^*)^{-1}}{q-1}+\varphi^{-1} R, \qquad \qquad
A_- =  \frac{(A^*)^{-1}}{q-1} -\varphi (A^*)^{-1} L
\end{align*}
and note that
$A =  \varphi A_+- A_-$.
We show that
\begin{align*}
&q A_-  A^* - A^* A_- = I, \qquad \qquad  q A^* A_+  -A_+ A^*= I, \\
&   q A_+ A_-   -A_- A_+ = \frac{q^N I}{q-1}.
\end{align*}
Using these relations, we show that the standard module $V$ becomes a $U_{q^{1/2}} (\mathfrak{sl}_2)$-module on which the generators $\mathcal X$, $\mathcal Y$, $\mathcal Y^{-1}$, $\mathcal Z$ 
act as follows:
\begin{align*} 
\begin{tabular}[t]{c|cccc}
{\rm generator $g$ }& $\mathcal X$ & $\mathcal Y$ & $\mathcal Y^{-1}$ & $\mathcal Z$
 \\
 \hline
 {\rm action of $g$ on $V$} & $(q-1)q^{-N/2} A_-$ & $q^{N/2}A^*$ & $q^{-N/2} (A^*)^{-1}$ & $(q-1)q^{-N/2} A_+$
    \end{tabular}
\end{align*}
  \noindent We now describe our third main result.
  We show that
 each of  $A^+, A^-, A_+, A_-$ is diagonalizable, with eigenvalues $q^i/(q-1)$ $(0 \leq i \leq N)$.
 We display four bases for the standard module $V$, said to be split. We show that each split basis diagonalizes one of $A^+, A^-, A_+, A_-$.
 We  give the action of $A, A^*$ on each split basis. We show that on each split basis, one of $A, A^*$ acts in an upper triangular fashion
 and the other one acts in a lower triangular fashion.
  Using these actions, we argue that the eigenspace decompositions of $A^+, A^-, A_+, A_-$ are the analogs of the four split decompositions
   that appear in \cite{ds, kim1,kim2,qtet}.
 \medskip
 
 \noindent We just described our three main results. The following auxiliary results may be of independent interest.
   We show that $A, A^*$ satisfy the tridiagonal relations \cite[Definition~3.9]{qSerre}.
  We consider the subalgebra $T$ of ${\rm Mat}_X(\mathbb C)$ generated by $A, A^*$.
   We show that the standard module $V$ is  a direct sum of irreducible $T$-submodules. 
   For each irreducible $T$-submodule $W$ of $V$,
   we describe the action of  $A,A^*$  on $W$ in terms of Leonard systems
   of dual $q$-Krawtchouk type \cite[Example~20.7]{LSnotes}.
 \medskip

\noindent This paper is organized as follows.
Section 2 contains some preliminaries.
In Section 3 we recall the poset $L_N(q)$.
In Sections 4, 5 we discuss the matrix $A^*$ and its primitive idempotents.
In Section 6 we discuss the matrices $R, L$.
In Section 7 we consider the matrix $A$.
In Section 8 we discuss the algebra $T$ generated by $A, A^*$.
In Section 9 we describe the irreducible $T$-modules.
In Section 10 we show that $A$ is diagonalizable, and find its eigenvalues.
In Section 11 we describe the primitive idempotents of $A$.
In Section 12 we display a $Q$-polynomial structure for $L_N(q)$.
In Section 13 we  explain how $L_N(q)$ is related to $U_{q^{1/2}}(\mathfrak{sl}_2)$.
Section 14 is about the irreducible $T$-modules and Leonard systems of dual $q$-Krawtchouk type.
Section 15 is about $U_{q^{1/2}}(\mathfrak{sl}_2)$ and split decompositions.
\medskip

\noindent The main results of the paper are Theorems \ref{cor:AQpoly}, \ref{lem:U2module}, \ref{lem:U2module2},  \ref{lem:ECE}.

\section{Preliminaries}  We now begin our formal argument.
The following concepts and notation will be used throughout the paper. Let $\mathbb C$ denote the field of complex numbers.
 Every algebra discussed in this paper is
understood to be
associative, over $\mathbb C$, and have a multiplicative identity. A subalgebra has the same multiplicative identity as the parent algebra.
\medskip

\noindent
Let $X$ denote a nonempty finite set. 
Let ${\rm Mat}_X(\mathbb C)$ denote the algebra consisting of the matrices that have rows and columns indexed by $X$ and  all entries in $\mathbb C$.
The identity matrix in ${\rm Mat}_X(\mathbb C)$ is denoted by $I$.
Let $V=\mathbb C^X$ denote the vector space over $\mathbb C$ consisting of the column vectors that have coordinates indexed by $X$ and all entries in $\mathbb C$.  Note that ${\rm Mat}_X(\mathbb C)$ acts on $V$
by left multiplication. We call $V$ the {\it standard module}. We endow $V$ with a Hermitean form $\langle\,,\,\rangle$ such that $\langle u,v\rangle = u^t \overline v$ for all $u,v \in V$, where $t$ denotes transpose and $-$ denotes complex-conjugation.
For $x \in X$ let ${\hat x}$ denote the vector in $V$ that has $x$-coordinate $1$ and all other coordinates 0. The vectors $\lbrace {\hat x} \vert x \in X\rbrace$
form an orthonormal basis for $V$. A matrix $B \in {\rm Mat}_X(\mathbb C)$ is called {\it diagonalizable} whenever $V$ is spanned by the eigenspaces of $B$. Assume that $B$ is diagonalizable, and let $\theta$ denote an
eigenvalue of $B$. By the {\it primitive idempotent for $B$ and $\theta$} we mean  the matrix in ${\rm Mat}_X(\mathbb C)$ that acts as the identity on the $\theta$-eigenspace of $B$, and as zero on every other eigenspace of $B$.
\medskip

\noindent For a positive real number $\alpha$ let $\alpha^{\frac{1}{2}}$ denote the positive square root of $\alpha$.
\section{The projective geometry $L_N(q)$}

\noindent In this section, we recall the projective geometry $L_N(q)$ and discuss its basic properties.
Background information about $L_N(q)$ can be found in \cite[Section~9.3]{bcn},  \cite[Chapter~1]{cameron},
 \cite{murali, murali2}, 
  \cite[Example 3.1(5) with $M=N$]{uniform},  \cite[Section~7]{LSintro}, \cite{Lnq}.

\begin{definition}\label{def:LNq} \rm Let $GF(q)$ denote a finite field with cardinality $q$. Fix an integer $N\geq 1$. Let $\bf V$ denote a vector space over $GF(q)$
that has dimension $N$. Let the set $X$ consist of the subspaces of $\bf V$. 
The zero subspace of $\bf V$ is denoted by $\bf 0$. Define a partial order $\leq$ on $X$ such that for $x,y \in X$,
$x\leq y$ whenever $x \subseteq y$. The poset $X, \leq $ is denoted $L_N(q)$ and called a {\it projective geometry}.  An element of $X$ is  called a {\it vertex}.
\end{definition}
\noindent For the rest of this paper, we refer to the poset $L_N(q)$ in Definition \ref{def:LNq}.
\medskip

\noindent The following result is well known; see for example \cite[p.~47]{axler}.
\begin{lemma} \label{lem:modular} {\rm (See \cite[p.~47]{axler}.)}  For $x,y \in X$ we have
\begin{align*}
{\rm dim}\,x
+{\rm dim}\,y =
{\rm dim}\,(x \cap y)
+ {\rm dim}\, (x+y).
\end{align*}
\end{lemma}

\noindent  We will use the following concepts and notation.
 Let $x,y \in X$. We write $x<y$ whenever $x \leq y $ and $x \not=y$. We say that $y$ {\it covers} $x$ whenever $x <y$ and there does not
exist $z \in X$ such that $x <z<y$. Note that $y$ covers $x$ if and only if $x \leq y$ and ${\rm dim}\, y - {\rm dim}\, x = 1$. 
We say that $x,y$ are {\it adjacent} whenever one of $x,y$ covers the other one. 
%%The adjacency relation turns $X$ into an undirected graph
%%that we denote by $\Gamma$. 
%Let $\partial $ denote the path-length distance function for $\Gamma$. One checks that $\Gamma$ has diameter $N$. For  $x \in X$ and $0 \leq i \leq N$
%define the set $\Gamma_i(x)= \lbrace y \in X \vert \partial(x,y)=i\rbrace$. We abbreviate $\Gamma(x)=\Gamma_1(x)$. For notational convenience define $\Gamma_{-1}({\bf 0}) = \emptyset$ and $\Gamma_{N+1}({\bf 0}) = \emptyset$.
%\noindent We have a  comment. For $y \in X$ we have $\partial({\bf 0}, y) = {\rm dim}\,y$. Therefore
%\begin{align*}
%\Gamma_i({\bf 0}) = \lbrace y \in X \vert {\rm dim}\,y =i \rbrace, \qquad \qquad (0 \leq i \leq N).
%\end{align*}
 For an integer $n\geq 0$ define
\begin{align*} 
\lbrack n \rbrack_q = \frac{q^n-1}{q-1}.
\end{align*}
\noindent We further define
\begin{align*}
\lbrack n \rbrack^!_q = \lbrack n \rbrack_q \lbrack n-1 \rbrack_q \cdots \lbrack 2 \rbrack_q \lbrack 1 \rbrack_q.
\end{align*}
We interpret $\lbrack 0 \rbrack^!_q=1$. For $0 \leq i \leq n$ define the $q$-binomial coefficient
\begin{align*}
\binom{n}{i}_q = \frac{\lbrack n \rbrack^!_q}{\lbrack i \rbrack^!_q \lbrack n-i \rbrack^!_q}.
\end{align*}

\begin{lemma} \label{lem:yi} Let $y \in X$ and write $i = {\rm dim}\,y$. Then:
\begin{enumerate}
\item[\rm (i)] $y$ covers exactly $\lbrack i \rbrack_q$ vertices;
\item[\rm (ii)] $y$ is covered by exactly $\lbrack N-i \rbrack_q$ vertices.
\end{enumerate}
\end{lemma}
\begin{proof} By routine combinatorial counting; see for example \cite[Section~9.3]{bcn}.
\end{proof}

\noindent The following result is well known; see for example \cite[Lemma~9.3.2]{bcn}.

\begin{lemma} \label{lem:size} {\rm (See \cite[Lemma~9.3.2]{bcn}.)} For $0 \leq i \leq N$,
the number of vertices that have dimension $i$ is equal to $\binom{N}{i}_q$.
\end{lemma}
\begin{proof} By Lemma \ref{lem:yi} and induction on $i$.
\end{proof}

\section{The matrices $\lbrace E^*_i \rbrace_{i=0}^N$ and the algebra $M^*$}
\noindent We continue to discuss the poset $L_N(q)$. In this section, we introduce the  matrices $\lbrace E^*_i \rbrace_{i=0}^N$ and the algebra $M^*$.

\begin{definition} \label{def:Eis} \rm (See \cite[Section~1]{uniform}.)  For $0 \leq i \leq N$ we define a diagonal matrix $E^*_i \in {\rm Mat}_X(\mathbb C)$ with $(y,y)$-entry
\begin{align*}
(E^*_i)_{y,y} = \begin{cases} 1, & \hbox{if ${\rm dim}\,y = i$};\\
                                              0, & \hbox{if ${\rm dim}\,y \not= i$}
                       \end{cases}                     
                       \qquad \qquad (y \in X).
\end{align*}
\noindent For notational convenience, define $E^*_{-1}=0$ and $E^*_{N+1}=0$.
\end{definition}

\begin{lemma} \label{lem:EE} We have
\begin{align*}
&E^*_i E^*_j = \delta_{i,j} E^*_i \qquad  (0 \leq i,j\leq N), \qquad \qquad I = \sum_{i=0}^N E^*_i.
\end{align*}
\end{lemma}
\begin{proof} By Definition \ref{def:Eis}.
\end{proof}

\begin{definition} \label{def:Ms} \rm (See \cite[p.~378]{tSub1}.)
 By Lemma \ref{lem:EE}, the matrices $\lbrace E^*_i \rbrace_{i=0}^N$ form a basis for a commutative subalgebra $M^*$ of ${\rm Mat}_X(\mathbb C)$. 
We call $M^*$ the {\it dual adjacency algebra} of $L_N(q)$ with respect to the vertex $\bf 0$.
\end{definition}

\begin{lemma} \label{lem:Eis}  For $0 \leq i \leq N$ and $y \in X$,
\begin{align*}
E^*_i {\hat y} = \begin{cases} {\hat y}, & \hbox{if ${\rm dim}\,y = i$};\\
                                              0, & \hbox{if ${\rm dim}\,y \not= i$}.
                       \end{cases}                    
\end{align*}
\end{lemma}
\begin{proof} By Definition \ref{def:Eis}.
\end{proof}

\begin{lemma} For $0 \leq i \leq N$ we have
\begin{align} \label{eq:cone}
E^*_iV = {\rm Span} \lbrace {\hat y} \vert y \in X, \; {\rm dim}\,y=i \rbrace
\end{align}
and
\begin{align}
{\rm dim}\, E^*_iV = \binom{N}{i}_q. \label{eq:ctwo}
\end{align}
\end{lemma}
\begin{proof} Assertion \eqref{eq:cone} follows  from Lemma \ref{lem:Eis}. 
Assertion \eqref{eq:ctwo} follows from \eqref{eq:cone} and Lemma \ref{lem:size}.
\end{proof}

\begin{lemma} \label{lem:DS} We have
\begin{align*}
V = \sum_{i=0}^N E^*_iV \qquad \quad \hbox{\rm (orthogonal direct sum)}.
\end{align*}
\end{lemma}
\begin{proof} By \eqref{eq:cone} and since $\lbrace {\hat y} \vert y \in X\rbrace$ is an orthonormal basis for $V$.
\end{proof}

\section{The matrix $A^*$}
\noindent We continue to discuss the poset $L_N(q)$. In this section, we interpret the subspaces $E^*_iV$ $(0 \leq i \leq N)$ as the eigenspaces
for a certain matrix  $A^*$ that  generates $M^*$.
\begin{definition}\label{def:As} \rm {\rm (See \cite[Definition~5.7]{Lnq}.)} Define a diagonal matrix $A^* \in {\rm Mat}_X(\mathbb C)$ with $(y,y)$-entry
\begin{align*}
A^*_{y,y} = q^{-{\rm dim}\,y}, \qquad \qquad (y \in X).
\end{align*}
\end{definition}

\noindent In Section 12, we will see that  $A^*$ is a dual adjacency matrix for $L_N(q)$, in the sense of  \cite[Definition~20.6]{int}. 
\medskip

\noindent We have some comments about $A^*$. The matrix $A^*$ is invertible. We have
\begin{align}
A^* {\hat y} = q^{-{\rm dim}\,y} {\hat y}, \qquad \qquad (A^*)^{-1} {\hat y} = q^{{\rm dim}\,y} {\hat y},  \qquad \qquad (y \in X). \label{eq:AsAct}
\end{align}
We have 
\begin{align} \label{eq:sumAs}
A^*= \sum_{i=0}^N q^{-i} E^*_i, \qquad \qquad (A^*)^{-1} = \sum_{i=0}^N q^i E^*_i.
\end{align}
The eigenvalues of $A^*$ are $\lbrace q^{-i} \rbrace_{i=0}^N$.
For $0 \leq i \leq N$, $E^*_i$ is the primitive idempotent of $A^*$ for the eigenvalue $q^{-i}$, and $E^*_iV$ is the corresponding eigenspace.
For notational convenience,  abbreviate $\theta^*_i = q^{-i}$ for $0 \leq i \leq N$. By linear algebra,
\begin{align} \label{eq:Ais}
 E^*_i=\prod_{\stackrel{0 \leq j \leq N}{j \neq i}}
       \frac{A^*-\theta^*_jI}{\theta^*_i-\theta^*_j}, \qquad \qquad (0 \leq i \leq N).
\end{align}
\noindent Note that the algebra $M^*$ is generated by $A^*$.

\section{The matrices $R, L$}
\noindent We continue to discuss the poset $L_N(q)$. In this section, we introduce some matrices $R, L \in {\rm Mat}_X(\mathbb C)$. We describe
how $R, L$ act on $V$, and  how $A^*, R, L$ are related.

\begin{definition} \label{def:RL} \rm  (See \cite[Section~3]{dickieTer}.) We define the matrices $R, L \in {\rm Mat}_X(\mathbb C)$ as follows. For $y,z \in X$ their $(y,z)$-entries are
\begin{align*}
R_{y,z} = \begin{cases} 1, & \hbox{if $y$ covers $z$};\\
                                              0, & \hbox{if $y$ does not cover $z$},
                       \end{cases}
             \qquad 
 L_{y,z} = \begin{cases} 1, & \hbox{if $z$ covers $y$};\\
                                              0, & \hbox{if $z$ does not cover $y$}.
                       \end{cases} 
\end{align*}
Note that $R^t=L$. We call $R$ (resp. $L$) the {\it raising matrix} (resp. {\it lowering matrix}) of $L_N(q)$.
\end{definition}

\begin{lemma} \label{lem:RLact} For $y \in X$ we have
\begin{align*}
R {\hat y} = \sum_{z \;{\rm covers}\;  y} {\hat z}, \qquad \qquad L {\hat y} = \sum_{y \;{\rm covers}\;  z} {\hat z}.
\end{align*}
\end{lemma}
\begin{proof} By Definition  \ref{def:RL} and the construction.
\end{proof}

\begin{lemma} \label{lem:RaiseLower} We have
\begin{align*}
&R E^*_iV \subseteq E^*_{i+1} V \qquad  (0 \leq i \leq N-1), \qquad \qquad RE^*_NV=0, \\
& L E^*_iV \subseteq E^*_{i-1}V \qquad  (1 \leq i \leq N), \qquad \qquad LE^*_0V=0.
\end{align*}
\end{lemma}
\begin{proof} By \eqref{eq:cone} 
and Lemma \ref{lem:RLact}.
\end{proof}

\begin{lemma} \label{lem:RLE} We have
\begin{align*}
&R E^*_i = E^*_{i+1} R \not=0 \qquad (0 \leq i \leq N-1), \qquad \quad RE^*_N=0, \quad E^*_0R=0, \\
&L E^*_i = E^*_{i-1} L \not=0\qquad (1 \leq i \leq N), \qquad \qquad LE^*_0=0, \quad E^*_NL=0.
\end{align*}
\end{lemma}
\begin{proof} The equalities are from Lemmas \ref{lem:EE}, \ref{lem:RaiseLower}. The inequalities follow from Lemma \ref{lem:yi}. 
\end{proof}

\noindent Next, we describe how $A^*, R, L$ are related.
\begin{lemma} \label{lem:AsRL}  We have
\begin{align*}
&A^* L = q L A^*, \qquad \qquad A^* R = q^{-1} R A^*, \\
 & LR-RL = \frac{q^N A^* - (A^*)^{-1}}{q-1}.
\end{align*}
\end{lemma}
\begin{proof} The first two relations follow from Lemma \ref{lem:RaiseLower} and the comments below Definition \ref{def:As}.
To get the third relation, use Lemmas \ref{lem:modular}, \ref{lem:yi}, \ref{lem:RLact}.
\end{proof}

\section{The matrix $A$ and the algebra $M$}
\noindent We continue to discuss the poset $L_N(q)$. In this section, we introduce the matrix $A$ and the algebra $M$.
 The matrix $A$ will depend on a positive parameter $\varphi$ contained in the field $\mathbb R$ of real numbers.
  The choice of $\varphi$ is arbitrary.
We remark that $A$ generalizes the weighted adjacency matrix for $L_N(q)$ given in \cite[Section~1]{murali}.

\begin{definition}\label{def:weightedA} \rm Pick $\varphi\in \mathbb R$ such that $\varphi>0$.
Define a  matrix $A \in {\rm Mat}_X(\mathbb C)$ with $(y,z)$-entry
\begin{align*}
A_{y,z} = \begin{cases} 1, & \mbox{if $y$ covers $z$}; \\
                                       \varphi q^{{\rm dim} \,y}, & \mbox{if $z$ covers $y$}; \\
                                        \frac{\varphi-1}{q-1} q^{{\rm dim}\,y}, & \mbox{if $y=z$}; \\
                                         0, &     \mbox{if $y,z$ are distinct and nonadjacent}.
                  \end{cases} \qquad  (y,z \in X).
\end{align*}
\end{definition}

\begin{remark}\rm For $\varphi=1$, the matrix $A$ in Definition \ref{def:weightedA} appears in \cite[Section~1]{murali}.
\end{remark}

\begin{definition}\label{def:M} \rm Let $M$ denote the subalgebra of ${\rm Mat}_X(\mathbb C)$ generated by $A$.
We call $M$ the {\it adjacency algebra} for $L_N(q)$ associated with $A$.
\end{definition}

\noindent Next, we express $A$ in terms of $R, L, A^*$.

\begin{lemma} \label{lem:Ashape} We have
\begin{align*}
A = R + \varphi (A^*)^{-1} L+ \frac{\varphi-1}{q-1}(A^*)^{-1}.
\end{align*}
\end{lemma}
\begin{proof} To verify this equation, compute the entries  of each side using Definitions \ref{def:As}, \ref{def:RL}, \ref{def:weightedA}.
\end{proof}
\noindent Next, we express $R, L$ in terms of $A$ and $\lbrace E^*_i\rbrace_{i=0}^N$.
\begin{lemma} \label{eq:RLpoly} We have
\begin{align*}
R = \sum_{i=1}^N E^*_i A E^*_{i-1}, \qquad \qquad L= \varphi^{-1}  \sum_{i=1}^N q^{1-i}E^*_{i-1} A E^*_i.
\end{align*}
\end{lemma}
\begin{proof} To verify these equations, eliminate $A$ using Lemma \ref{lem:Ashape}, and evaluate the results using \eqref{eq:sumAs} along with
Lemmas  \ref{lem:EE},  \ref{lem:RLE}.
\end{proof}

\noindent Next, we describe how $A$ and $\lbrace E^*_i \rbrace_{i=0}^N$ are related.

\begin{lemma}\label{lem:triples} For $0 \leq i,j\leq N$ we have
\begin{align*}
E^*_i A E^*_j = \begin{cases} 0, & \mbox{if $\vert i-j\vert >1$}; \\
                                            \not=0, & \mbox{if $\vert i-j\vert=1$}.
                         \end{cases}
\end{align*}
\end{lemma}
\begin{proof} By Definitions \ref{def:Eis}, \ref{def:weightedA} and the construction.
 %%%\eqref{eq:sumAs} and Lemmas \ref{lem:EE}, \ref{lem:RLE}, \ref{lem:Ashape}. 
\end{proof}

\begin{proposition} For $0 \leq i \leq N$,
\begin{align*}
A E^*_iV \subseteq E^*_{i-1}V + E^*_iV + E^*_{i+1}V.
\end{align*}
\end{proposition}
\begin{proof} By $I=\sum_{i=0}^N E^*_i$ and Lemma \ref{lem:triples},
\begin{align*}
A E^*_i &= (E^*_0 + E^*_1+\cdots + E^*_N) A E^*_i \\
              &= E^*_{i-1} A E^*_i + E^*_iAE^*_i + E^*_{i+1}A E^*_i.
\end{align*}
The result follows.
\end{proof}

\section{The algebra $T$}
\noindent We continue to discuss the poset $L_N(q)$. In this section, we recall the algebra $T$ from  \cite[Section~1]{uniform}. We give several generating sets for $T$. We also describe how $A, A^*$ are related.

\begin{definition}\rm \label{def:T} {\rm (See \cite[Section~1]{uniform}.)} Let $T$ denote the subalgebra of ${\rm Mat}_X(\mathbb C)$ generated by $R, L, \lbrace E^*_i\rbrace_{i=0}^N$.
\end{definition}

\begin{lemma} \label{lem:TAA} The algebra $T$ is generated by $A, A^*$.
\end{lemma}
\begin{proof}  By \eqref{eq:sumAs} and Lemma \ref{lem:Ashape}, we have $A, A^* \in T$. By \eqref{eq:Ais} and Lemma \ref{eq:RLpoly}, each of $\lbrace E^*_i \rbrace_{i=0}^N$, $R$, $L$ is a polynomial in $A, A^*$.
\end{proof}

\begin{corollary} The algebra $T$ is generated by $M, M^*$.
\end{corollary}
\begin{proof} The algebras $M$ and $M^*$ are generated by $A$ and $A^*$, respectively. The result follows in view of Lemma \ref{lem:TAA}.
\end{proof}

\noindent Next, we describe how $A, A^*$ are related. For notational convenience, abbreviate
\begin{align} 
\beta =  q+q^{-1}.      \label{eq:beta}
\end{align}
\begin{proposition} \label{prop:qSerre} We have
 \begin{align}
 A^{*3} A - (\beta+1) A^{*2} A A^* + (\beta+1) A^* A A^{*2} - A A^{*3} &=0. \label{eq:Serre}
\end{align}
\end{proposition}
\begin{proof} Let $C$ denote the expression on the left in \eqref{eq:Serre}. We show that $C=0$. Observe that
\begin{align*}
C = I C I =  \Biggl( \sum_{i=0}^N E^*_i \Biggr)  C          \Biggl( \sum_{j=0}^N E^*_j \Biggr)  =\sum_{i=0}^N \sum_{j=0}^N E^*_i C E^*_j.
\end{align*}
For $0 \leq i,j\leq N$ we show that $E^*_i C E^*_j=0$.  Using $E^*_i A^*=\theta^*_i E^*_i$ and $A^* E^*_j = \theta^*_j E^*_j$ along with \eqref{eq:beta}, we obtain
\begin{align*}
E^*_i C E^*_j &= E^*_i A E^*_j \Bigl(\theta^{* 3}_i - (\beta+1) \theta^{*2}_i \theta^*_j +(\beta+1) \theta^*_i \theta^{*2}_j - \theta^{*3}_j \Bigr)  \nonumber \\
                       &= E^*_i A E^*_j (\theta^*_i - \theta^*_j)(\theta^*_i -q \theta^*_j)(\theta^*_i - q^{-1} \theta^*_j).
\end{align*}
We examine the factors in the previous line.
If  $\vert i-j \vert >1$, then $E^*_i A E^*_j=0$ by Lemma \ref{lem:triples}. If $i-j=1$, then $\theta^*_i = q^{-1}\theta^*_j$.
If $i-j=-1$, then $\theta^*_i = q\theta^*_j$.
If $i-j=0$, then of course $\theta^*_i = \theta^*_j$.
In every case $E^*_i C E^*_j=0$. The result follows.
\end{proof}

\noindent For notational convenience, abbreviate
\begin{align} 
      \varrho = \varphi q^{N-2}(q+1)^2.        \label{eq:rhoval}
%%%\label{eq:beta}
\end{align}

\noindent  In a moment, we will show that
\begin{align*}
 A^3 A^* - (\beta+1) A^2 A^*& A + (\beta+1) A A^* A^2 -A^* A^3 =\varrho (A A^*-A^* A), 
 \end{align*}
where $\beta=q+q^{-1}$ and $\varrho$ is from    \eqref{eq:rhoval}. The following definition and lemma will simplify the calculations.
\begin{definition}\rm \label{def:Apm} Define
\begin{align*}
A^+ =  \frac{(A^*)^{-1}}{q-1}-R, \qquad \qquad
A^- =  \frac{(A^*)^{-1}}{q-1} +(A^*)^{-1} L.
\end{align*}
\end{definition}
\begin{lemma}\label{lem:UqRel} We have
\begin{align}
A =  \varphi A^{-}-A^+. \label{eq:AAA}
\end{align}
Moreover
\begin{align}
&q A^-  A^* - A^* A^- = I, \qquad \qquad  q A^* A^+  -A^+ A^*= I, \label{eq:UqOne}\\
&   q A^+ A^-   -A^- A^+ = \frac{q^N I}{q-1}. \label{eq:UqTwo}
\end{align}
\end{lemma}
\begin{proof} To verify \eqref{eq:AAA}, eliminate $A^+, A^-$  using Definition \ref{def:Apm} and evaluate the result using Lemma \ref{lem:Ashape}.
To verify \eqref{eq:UqOne} and \eqref{eq:UqTwo}, eliminate $A^+, A^-$  using Definition \ref{def:Apm} and evaluate the result using Lemma \ref{lem:AsRL}.
\end{proof}

\begin{proposition} \label{thm:TDrel} We have
\begin{align}
 A^3 A^* - (\beta+1) A^2 A^*& A + (\beta+1) A A^* A^2 -A^* A^3 =\varrho (A A^*-A^* A), \label{eq:TDrel}
 \end{align}
where $\beta=q+q^{-1}$ and $\varrho$ is from    \eqref{eq:rhoval}.
\end{proposition}
\begin{proof}  
To verify  \eqref{eq:TDrel}, eliminate $A$ using \eqref{eq:AAA} and evaluate the result using \eqref{eq:UqOne}, \eqref{eq:UqTwo}.
\end{proof}

\begin{remark}  \rm The relations \eqref{eq:Serre}, \eqref{eq:TDrel} are a special case of the tridiagonal relations, see \cite[Lemma~5.4]{tSub3} and  \cite[Definition~3.9]{qSerre}.
\end{remark}

\section{The irreducible $T$-modules}

\noindent We continue to discuss the poset $L_N(q)$. Recall the algebra $T$ from Definition \ref{def:T}.
In this section, we describe the irreducible $T$-modules.
\medskip

\noindent
 Let $W$ denote  a $T$-module. Then $W$ is called 
{\it irreducible} whenever $W \not=0$ and $W$ does not contain a $T$-submodule besides $0$ and $W$.

\begin{lemma} \label{lem:ods} {\rm (See \cite[Theorem~2.5(1)]{uniform}.)} The standard module $V$ is an orthogonal direct sum of irreducible $T$-submodules.
\end{lemma}
\begin{proof} Each of the $T$-generators $R, L, \lbrace E^*_i \rbrace_{i=0}^N$ has real entries. We have $R^t=L$. Moreover, $E^*_i$ is symmetric for $0 \leq i \leq N$.
By these comments,  $T$ is closed under the conjugate-transpose map.
The result is a routine consequence of this.
\end{proof}

\noindent Let $W$ denote an irreducible $T$-submodule of $V$. Observe that $W$ is the orthogonal direct sum of the nonzero subspaces among $\lbrace E^*_i W \vert 0 \leq i \leq N\rbrace$.
By the {\it endpoint} of $W$, we mean
\begin{align*}
{\rm min}\lbrace i \vert 0 \leq i \leq N, E^*_i W \not= 0\rbrace.
\end{align*}
\noindent By the {\it diameter} of $W$, we mean
\begin{align*}
\bigl\vert \lbrace i \vert 0 \leq i \leq N, \; E^*_iW \not=0\rbrace \bigr\vert -1.
\end{align*}
We say that $W$ is {\it thin} whenever  ${\rm dim}\,E^*_iW\leq 1$  for $0 \leq i \leq N$.

\begin{lemma} \label{lem:Wshape} {\rm (See \cite[Theorem~3.3(5)]{uniform}.)} Let $W$ denote an irreducible $T$-submodule of $V$. Then $W$ is thin. The endpoint $r$ and diameter $d$ of $W$ satisfy
$0 \leq r \leq N/2$ and $d=N-2r$.
\end{lemma}
\noindent We will be discussing matrix representations, using the notational conventions of \cite[p.~5]{nomSpinModel}.
\begin{lemma}
\label{lem:Wrep} {\rm (See \cite[Theorems 2.5 and 3.3(5)]{uniform}.)} Let $W$ denote an irreducible $T$-submodule of $V$. Let $r$ denote the endpoint of $W$, and recall the diameter $d=N-2r$. Then there exists a basis
for $W$ with respect to which the matrices representing  $A^*, R, L$ are
\begin{align*}
&A^*: \quad {\rm diag} (q^{-r}, q^{-r-1}, \ldots, q^{-r-d} ), \\
&R: \quad \begin{pmatrix} 0 & 0 & &&&{\bf 0}  \\
                                             1 & 0 & 0 & &&\\
                                              & 1 & \cdot & \cdot && \\
                                              && \cdot &\cdot &\cdot &\\
                                              &&&\cdot &\cdot &0\\
                                              {\bf 0} &&&&1&0
                                             \end{pmatrix},   \\        
     &L: \quad \begin{pmatrix} 0 & x_1(W) & &&&{\bf 0}  \\
                                             0 & 0 & x_2(W) & &&\\
                                              & 0 & \cdot & \cdot && \\
                                              && \cdot &\cdot &\cdot &\\
                                              &&&\cdot &\cdot &x_d(W)\\
                                              {\bf 0} &&&&0&0
                                             \end{pmatrix},   \\                                                                                                                    
\end{align*}
\noindent where
\begin{align*}
x_i(W) &=  q^{1-i} \, \frac{q^i-1}{q-1} \, \frac{q^{N-r}-q^{i+r-1}}{q-1}   \qquad \quad (1 \leq i \leq d).
\end{align*}
\end{lemma} 
\noindent The concept of $T$-module isomorphism is discussed in \cite[Section~1]{uniform}.

\begin{lemma} \label{lem:mult} {\rm (See \cite[Theorem~3.3(5)]{uniform}.)}  Let $W$ denote an irreducible $T$-submodule of $V$. Then $W$ is determined up to $T$-module isomorphism by its endpoint $r$.
For $0 \leq r \leq N/2$ let ${\rm mult}_r$ denote the multiplicity with which the irreducible $T$-submodule with endpoint $r$ appears in $V$. 
Then 
\begin{align}
{\rm mult}_r = \binom{N}{r}_q-\binom{N}{r-1}_q  \qquad  (1 \leq r \leq N/2), \qquad \quad {\rm mult}_0=1.  \label{eq:mform}
\end{align}
\end{lemma}
\begin{lemma} \label{lem:formula} For $0 \leq i \leq N/2$ we have
\begin{align*}
{\rm mult}_0 + {\rm mult}_1 + \cdots + {\rm mult}_i = \binom{N}{i}_q.
\end{align*}
\end{lemma}
\begin{proof} Immediate from \eqref{eq:mform}.
\end{proof}

%%%%%%%%%%%%%

\section{The eigenvalues of $A$}

\noindent We continue to discuss the poset $L_N(q)$. In this section, we examine the matrix $A$ from Definition \ref{def:weightedA}.
We show that $A$ is diagonalizable, and we compute its eigenvalues. For each eigenvalue of $A$, we give
the dimension of the corresponding eigenspace.
\medskip

\noindent The following result is a variation on \cite[Section~1]{murali}.

\begin{lemma} \label{thm:Adiag} The matrix $A$ is diagonalizable. 
\end{lemma}
\begin{proof} Define a sequence of scalars $\lbrace s_i \rbrace_{i=0}^N$ such that $s_0=1$ and
\begin{align*}
s_{i+1} = s_i q^{\frac{i}{2}} \varphi^{\frac{1}{2}} \qquad \qquad (0 \leq i \leq N-1).
\end{align*}
For $0 \leq i \leq N$ we have $0 \not=s_i \in \mathbb R$, because  $q$ and $\varphi $ are positive real numbers.
Define a diagonal matrix $D \in {\rm Mat}_X(\mathbb C)$ with $(y,y)$-entry
\begin{align*}
D_{y,y} = s_{{\rm dim}\,y}   \qquad \qquad (y \in X).
\end{align*}
\noindent The matrix  $D$ is invertible. Using Lemma  \ref{lem:Ashape} and $R^t=L$ we obtain $A^t = D^2 A D^{-2}$. Therefore $DAD^{-1}$ is symmetric.
The matrix $DAD^{-1}$ is symmetric with all entries real, so $DAD^{-1}$ is diagonalizable by linear algebra. It follows that $A$ is diagonalizable.
\end{proof}

\begin{definition} \label{def:theta} \rm Define the scalars
\begin{align*}
%%%\theta_i = \frac{(\lambda+1)q^{N-i} -\lambda -q^i}{q-1} \qquad \qquad (0 \leq i \leq N).
\theta_i = \frac{\varphi q^{N-i}-q^i}{q-1}  \qquad \qquad (0 \leq i \leq N).
\end{align*}
\end{definition}

\begin{lemma} \label{lem:thetadiff} For $0 \leq i,j\leq N$ we have
\begin{align*}
\theta_i-\theta_j = \Bigl( 1+\varphi q^{N-i-j}\Bigr) \frac{ q^j-q^i}{q-1}.
\end{align*}
\end{lemma}
\begin{proof}  Use Definition \ref{def:theta}.
\end{proof}

\begin{lemma} The scalars  $\lbrace \theta_i \rbrace _{i=0}^N$ are mutually distinct.
\end{lemma}
\begin{proof} By Lemma \ref{lem:thetadiff}.
\end{proof}

\noindent Our next  goal is twofold. We will show that $\lbrace \theta_i \rbrace _{i=0}^N$ are the eigenvalues of $A$. We will also show that for $0 \leq i \leq N$ the $\theta_i$-eigenspace of $A$ has dimension $\binom{N}{i}_q$.
\medskip

\begin{proposition} \label{thm:Wrep} Let $W$ denote an irreducible $T$-submodule of $V$, and let $r$ denote the endpoint of $W$. With respect to the
basis for $W$ from Lemma \ref{lem:Wrep}, the matrix representing $A$ is
\begin{align*}
&A: \quad \begin{pmatrix} a_0(W) & \xi_1(W) & &&&{\bf 0}  \\
                                             1 & a_1(W) & \xi_2(W) & &&\\
                                              & 1 & \cdot & \cdot && \\
                                              && \cdot &\cdot &\cdot &\\
                                              &&&\cdot &\cdot &\xi_d(W)\\
                                              {\bf 0} &&&&1&a_d(W)
                                             \end{pmatrix},                                                                               
\end{align*}
\noindent where
\begin{align*}
a_i(W) &= \frac{\varphi -1}{q-1} q^{i+r}  \qquad \qquad (0 \leq i \leq d), \\
\xi_i(W) &= \varphi q^r \frac{q^i-1}{q-1} \, \frac{q^{N-r}-q^{i+r-1}}{q-1} \qquad \quad (1 \leq i \leq d).
\end{align*}
\end{proposition}
\begin{proof} By Lemmas \ref{lem:Ashape}, \ref{lem:Wrep} and the observation
\begin{align*}
\xi_i(W)= \varphi q^{r+i-1} x_i(W) \qquad \qquad (1 \leq i \leq d).
\end{align*}
\end{proof}

\noindent Let $W$  denote a finite-dimensional vector space. A linear map $B:W\to W$ is called 
 {\it multiplicity-free} whenever $W$ is spanned by the eigenspaces of $B$, and each of these eigenspaces has dimension one.

\begin{proposition} \label{prop:Weigval} Let $W$ denote an irreducible $T$-submodule of $V$, and let $r$ denote the endpoint of $W$.
Then the action of $A$ on $W$ is multiplicity-free, with eigenvalues $\lbrace \theta_{i} \rbrace_{i=r}^{N-r}$. 
\end{proposition}
\begin{proof} Recall the diameter $d=N-2r$. Define a $(d+1) \times (d+1)$ matrix
\begin{align*}
&B= \begin{pmatrix} a_0 & b_0 & &&&{\bf 0}  \\
                                             c_1 & a_1 & b_1 & &&\\
                                              & c_2 & \cdot & \cdot && \\
                                              && \cdot &\cdot &\cdot &\\
                                              &&&\cdot &\cdot &b_{d-1}\\
                                              {\bf 0} &&&&c_d &a_d
                                             \end{pmatrix},                                                                               
\end{align*}
where
\begin{align*}
 c_i & = q^r \frac{q^i-1}{q-1}  \qquad \qquad (1 \leq i \leq d), \\
 a_i &= \frac{\varphi-1}{q-1}q^{i+r}   \qquad \quad (0 \leq i \leq d), \\
 b_i &= \varphi \frac{q^{N-r}-q^{i+r}}{q-1}\qquad \quad (0 \leq i \leq d-1).
\end{align*}
The matrix $B$ comes up in the description of a Leonard system $\Phi$ of dual $q$-Krawtchouk type \cite[Example~20.7]{LSnotes}, where the parameters
 $h(\Phi), s(\Phi), \theta_0(\Phi), h^*(\Phi), \theta^*_0(\Phi)$ from  \cite[Example~20.7]{LSnotes}
are given by
\begin{align*}
& h(\Phi) = \frac{\varphi q^{N-r}}{q-1}, \qquad \qquad s(\Phi) = - \varphi^{-1} q^{2r-N-1}, \\
& \theta_0(\Phi) = \frac{\varphi q^{N-r}-q^r}{q-1},  \qquad \quad h^*(\Phi)=q^{-r}, \qquad \quad \theta^*_0(\Phi) = q^{-r}.
\end{align*}
According to  \cite[Example~20.7]{LSnotes}, the matrix $B$ is multiplicity-free with eigenvalues
\begin{align*}
\theta_i(\Phi)= \theta_0(\Phi) + h(\Phi)(1-q^i)(1-s(\Phi)q^{i+1})q^{-i} = \theta_{i+r} \qquad \quad (0 \leq i \leq d).
\end{align*}
We have shown that $B$ is multiplicity-free with eigenvalues $\lbrace \theta_i \rbrace_{i=r}^{N-r}$.
Let $\mathcal B$ denote the matrix that represents $A$ in Proposition \ref{thm:Wrep}. The matrices $\mathcal B, B$ are related as follows.
Note that $a_i(W)=a_i$  $(0 \leq i \leq d)$ and $\xi_i(W)=c_i b_{i-1}$ $(1 \leq i \leq d)$. For $0 \leq i \leq d$, define $\eta_i = c_1 c_2 \cdots c_i$ and note that $\eta_i \not=0$.
 By matrix multiplication,
\begin{align*}
B \, {\rm diag}\bigl(\eta_0, \eta_1, \ldots, \eta_d \bigr) = 
{\rm diag} \bigl(\eta_0, \eta_1, \ldots, \eta_d \bigr) {\mathcal B}.
\end{align*}
Therefore, the matrices $\mathcal B, B$ have the same characteristic polynomial. By these comments, $\mathcal B$ is multiplicity-free with eigenvalues $\lbrace \theta_{i} \rbrace_{i=r}^{N-r}$. The result follows.
\end{proof}

\begin{proposition} \label{thm:Aspec} The eigenvalues of $A$ are $\lbrace \theta_i \rbrace_{i=0}^N$.
For $0 \leq i \leq N$ the $\theta_i$-eigenspace of $A$ has dimension $\binom{N}{i}_q$.
\end{proposition}
\begin{proof} By Lemma \ref{lem:ods}, the $T$-module $V$ is an orthogonal direct sum of irreducible $T$-submodules:
\begin{align}
\label{eq:VW}
V = \sum_{W} W.
\end{align}
For each summand $W$ in \eqref{eq:VW}, the endpoint of $W$ is nonnegative and at most $N/2$.
 For $0 \leq r \leq N/2$, the number of summands $W$ in \eqref{eq:VW} that have endpoint $r$
 is equal to  ${\rm mult}_r $. For each such $W$ the action of $A$ on $W$ is multiplicity-free, with
eigenvalues $\lbrace \theta_i \rbrace_{i=r}^{N-r}$.
The result follows in view of  Lemma \ref{lem:formula}.
\end{proof}

\section{The matrices $\lbrace E_i \rbrace_{i=0}^N$}

\noindent We continue to discuss the poset $L_N(q)$. In this section, we examine the primitive idempotents $\lbrace E_i\rbrace_{i=0}^N$ of $A$.

\begin{definition}\label{def:ei} \rm For $0 \leq i \leq N$ let $E_i$ denote the primitive idempotent of $A$ for the
eigenvalue $\theta_i$. For notational convenience, define $E_{-1}=0$ and $E_{N+1}=0$.
\end{definition}

\begin{lemma}\label{lem:AE} The matrices $\lbrace E_i \rbrace_{i=0}^N$ form a basis for $M$.
We have  $A= \sum_{i=0}^N \theta_i E_i$ and
\begin{align} \label{eq:EEm}
&E_i E_j = \delta_{i,j} E_i \qquad  (0 \leq i,j\leq N), \qquad \qquad I = \sum_{i=0}^N E_i.
\end{align}
Moreover
\begin{align} \label{eq:AEi}
 E_i=\prod_{\stackrel{0 \leq j \leq N}{j \neq i}}
       \frac{A-\theta_jI}{\theta_i-\theta_j}, \qquad \qquad (0 \leq i \leq N).
\end{align}
\end{lemma}
\begin{proof} By Definition \ref{def:ei} and linear algebra.
\end{proof}

\begin{lemma} \label{lem:Com2}
We have
\begin{align*}
V = \sum_{i=0}^N E_iV  \qquad \quad \hbox{\rm (direct sum)}.
\end{align*}
For $0 \leq i \leq N$ the subspace $E_iV$ is the eigenspace of $A$ for the eigenvalue $\theta_i$. Moreover,
\begin{align*}
{\rm dim}\, E_iV = \binom{N}{i}_q.
\end{align*}
\end{lemma}
\begin{proof} By Proposition  \ref{thm:Aspec} and Definition \ref{def:ei}.
\end{proof}

\section{A $Q$-polynomial structure for $L_N(q)$}

\noindent  In this section, we display a $Q$-polynomial structure for $L_N(q)$. This structure will be obtained using $A$ and $A^*$. Our first step is to show that
\begin{align}
E_i A^* E_j = \begin{cases} 0, & \mbox{ if $\vert i-j\vert >1$}; \\
                                            \not=0, & \mbox{ if $\vert i-j\vert=1$}.
                         \end{cases}    \qquad \qquad (0 \leq i,j\leq N).     \label{eq:EAsE}
\end{align}
\noindent This will be established over the next two lemmas.
\begin{lemma}\label{lem:EAsEZ}
Let $0 \leq i,j\leq N$ such that $\vert i-j\vert >1$. Then $E_i A^* E_j=0$.
\end{lemma}
\begin{proof} Referring to \eqref{eq:TDrel}, let $C$ denote the left-hand side minus the right-hand side. So $C=0$.
 Using Definition \ref{def:theta} along with 
$E_i A = \theta_i E_i$ and $A E_j = \theta_j E_j$,  we obtain
\begin{align*}
0 &= E_i C E_j \\
    &= E_i A^* E_j \Bigl( \theta^3_i - (\beta+1) \theta^2_i \theta_j +(\beta+1)\theta_i \theta^2_j - \theta^3_j - \varrho(\theta_i - \theta_j)    \Bigr) \\
    &= E_i A^* E_j (\theta_i - \theta_j)\bigl( \theta^2_i - \beta \theta_i \theta_j + \theta^2_j-\varrho\bigr) \\
    & = E_i A^* E_j (\theta_i-\theta_j) \frac{ (q^i-q^{j+1})(q^i-q^{j-1})(1+\varphi q^{N-i-j+1})(1+\varphi q^{N-i-j-1})}{(q-1)^2}.
\end{align*}
In the previous expression, the coefficient of $E_i A^* E_j$ is equal to $\theta_i-\theta_j$ times a fraction. We have $\theta_i - \theta_j \not=0$, since $i \not=j$ and
$\lbrace \theta_\ell \rbrace_{\ell=0}^N$ are mutually distinct. The fraction has four factors in the numerator. The first two factors are nonzero, because $i-j\not\in \lbrace 1,-1\rbrace$ and $q$ is not a root
of unity. The last two factors are nonzero, because $\varphi,q $ are positive real numbers. By these comments, the coefficient of $E_i A^* E_j$ is nonzero. Therefore $E_iA^*E_j=0$.
\end{proof} 

\begin{lemma}\label{lem:Primary} Let $0 \leq i,j\leq N$ such that $\vert i-j \vert =1$. Then $E_iA^*E_j\not=0$.
\end{lemma}
\begin{proof} By Lemma \ref{lem:mult}, there exists an irreducible $T$-submodule $W$ of $V$ that has endpoint $0$. By Proposition \ref{prop:Weigval}, the action of $A$ on $W$ is multiplicity-free with eigenvalues $\lbrace \theta_\ell \rbrace_{\ell=0}^N$.
So $E_\ell W\not=0$ for $0 \leq \ell \leq N$. First assume that $i-j=1$. Then $E_i A^* E_j W \not=0$; otherwise $\sum_{\ell=0}^j E_{\ell} W$ is a nonzero $T$-submodule that is properly contained in
$W$, contradicting the assumption that the $T$-module $W$ is irreducible. Next assume that $j-i=1$. Then $E_i A^* E_j W \not=0$; otherwise $\sum_{\ell=j}^N E_{\ell} W$ is a nonzero $T$-submodule that is properly contained in
$W$, contradicting the assumption that the $T$-module $W$ is irreducible.  We have shown that  $E_i A^* E_j W \not=0$. Therefore, $E_i A^*E_j \not=0$.
\end{proof}

\noindent We have established   \eqref{eq:EAsE}. Next we consider the implications.

\begin{proposition} \label{prop:TTT} We have
\begin{align*}
A^* E_iV \subseteq E_{i-1}V + E_iV + E_{i+1}V \qquad \qquad (0 \leq i \leq N).
\end{align*}
\end{proposition}
\begin{proof} By $I = \sum_{i=0}^N E_i$ and \eqref{eq:EAsE},  
\begin{align*}
          A^* E_i &= (E_0 + E_1+ \cdots + E_N)A^* E_i \\
                       &= E_{i-1}A^* E_i + E_i A^* E_i + E_{i+1} A^* E_i.
\end{align*}
The result follows.
\end{proof}

\begin{definition}\label{def:dualAdj} \rm (See \cite[Definition~20.6]{int}.) The matrix $A^*$ is called a {\it dual adjacency matrix of $L_N(q)$} (with respect to the vertex $\bf 0$ and the
given ordering $\lbrace E_i \rbrace_{i=0}^N$) whenever  $A^*$ generates $M^*$ and
\begin{align*}
 A^* E_iV \subseteq E_{i-1}V + E_iV + E_{i+1}V \qquad \quad (0 \leq i \leq N).
\end{align*}
\end{definition}

\begin{lemma} \label{lem:DAdj} The matrix $A^*$ is a dual adjacency matrix of $L_N(q)$ with respect to the vertex $\bf 0$ and the ordering  $\lbrace E_i \rbrace_{i=0}^N$.
\end{lemma}
\begin{proof}  By the note below   \eqref{eq:Ais},  along with   Proposition \ref{prop:TTT} and Definition \ref{def:dualAdj}.
\end{proof}
\begin{definition}\label{def:EQpoly} \rm {\rm (See \cite[Definition~20.7]{int}.)} The ordering  $\lbrace E_i \rbrace_{i=0}^N$ is said to be {\it $Q$-polynomial with respect to the vertex $\bf 0$} whenever there exists
a dual adjacency matrix of $L_N(q)$ with respect to $\bf 0$ and the ordering $\lbrace E_i \rbrace_{i=0}^N$.
\end{definition} 

\begin{proposition} \label{thm:EQpoly} The ordering $\lbrace E_i \rbrace_{i=0}^N$ is $Q$-polynomial with respect to the vertex $\bf 0$.
\end{proposition}
\begin{proof} By Lemma \ref{lem:DAdj} and Definition \ref{def:EQpoly}.
\end{proof}

\begin{definition} \label{def:AQpoly} \rm {\rm (See \cite[Definition~20.8]{int}.)} The matrix $A$ is said to be {\it $Q$-polynomial with respect to the vertex $\bf 0$} whenever there exists an ordering
of the primitive idempotents of $A$ that is $Q$-polynomial with respect to $\bf 0$.
\end{definition}

\begin{theorem} \label{cor:AQpoly} The  matrix $A$ is $Q$-polynomial with respect to the vertex $\bf 0$.
\end{theorem}

\begin{proof} By Proposition \ref{thm:EQpoly} and Definition \ref{def:AQpoly}.
\end{proof}

\noindent For $\varphi=1$, Theorem \ref{cor:AQpoly}  is the same as
\cite[Theorem~8.4]{Lnq}.

\section{$L_N(q)$ and the quantum group $U_{q^{1/2}} (\mathfrak{sl}_2)$}

\noindent  We continue to discuss the poset $L_N(q)$. Lemma  \ref{lem:AsRL} contains three equations that show how $A^*, R, L$ are related. In \eqref{eq:UqOne}, \eqref{eq:UqTwo} 
there are three equations that show how $A^*, A^+, A^-$ are related. In this section, we will explain
what these six equations have to do with the quantum group $U_{q^{1/2}} (\mathfrak{sl}_2)$. The algebra  $U_{q^{1/2}} (\mathfrak{sl}_2)$
has two familiar presentations in the literature, called the Chevalley presentation \cite[p.~122]{Kassel} and the equitable presentation \cite[Theorem~2.1]{equit}.
As we will see, Lemma \ref{lem:AsRL}  describes  $U_{q^{1/2}} (\mathfrak{sl}_2)$ in the Chevalley presentation, and 
\eqref{eq:UqOne}, \eqref{eq:UqTwo} describe  $U_{q^{1/2}} (\mathfrak{sl}_2)$ in the equitable presentation. We will also see how Propositions  \ref{prop:qSerre}, \ref{thm:TDrel} follow from a result  \cite[Proposition~7.17]{bockting}
concerning  $U_{q^{1/2}} (\mathfrak{sl}_2)$. Before we get into the details, we would like to acknowledge that a connection between $L_N(q)$ and  $U_{q^{1/2}} (\mathfrak{sl}_2)$ in the Chevalley presentation was previously discussed in
 \cite[Section~7]{LSintro}.
 
 \begin{definition} \label{def:qGroup} \rm (See \cite[p.~122]{Kassel}.)
 The algebra $U_{q^{1/2}} (\mathfrak{sl}_2)$ is defined by generators $K,K^{-1}, E, F$ and the following relations:
 \begin{align*}
&\qquad \quad K K^{-1} = K^{-1} K = 1, \\
& K E =  q E K, \qquad \quad K F = q^{-1} F K, \\
& \qquad EF-FE = \frac{K-K^{-1}}{q^{1/2}-q^{-1/2}}.
 \end{align*}
 We call $K,K^{-1}, E, F$ the {\it Chevalley generators} for $U_{q^{1/2}} (\mathfrak{sl}_2)$.
 \end{definition}

\noindent Next, we interpret Lemma  \ref{lem:AsRL} in terms of the Chevalley generators for $U_{q^{1/2}} (\mathfrak{sl}_2)$.
Recall the standard module $V$ for $L_N(q)$.
\begin{lemma} \label{lem:Umodule} {\rm (See  \cite[Section~7]{LSintro}.)} The standard module $V$ becomes a $U_{q^{1/2}} (\mathfrak{sl}_2)$-module on which
the Chevalley generators $K,K^{-1}, E, F$  act as follows:
\begin{align*} 
\begin{tabular}[t]{c|cccc}
{\rm generator $g$ }& $K$ & $K^{-1}$ & $E$ & $F$
 \\
 \hline
 {\rm action of $g$ on $V$} & $q^{N/2} A^*$ & $q^{-N/2} (A^*)^{-1}$ & $ L$ & $q^{\frac{1-N}{2}}R$
    \end{tabular}
\end{align*}
\end{lemma} 
\begin{proof} Compare the relations in Lemma \ref{lem:AsRL} and Definition  \ref{def:qGroup}.
\end{proof} 

\noindent Next, we recall the equitable presentation of $U_{q^{1/2}}({\mathfrak{sl}_2})$.
\begin{lemma} \label{lem:Eq}
{\rm
(See \cite[Theorem~2.1]{equit},  \cite[Lemma~5.1]{uqLUS}.)}
The algebra 
$U_{q^{1/2}}({\mathfrak{sl}_2})$ has a presentation  with generators
$\mathcal X$, $\mathcal Y$, $\mathcal Y^{-1}$, $\mathcal Z$ and relations
$\mathcal Y\mathcal Y^{-1} = \mathcal Y^{-1}\mathcal Y = 1$,
\begin{align}
\frac{q\mathcal X \mathcal Y-\mathcal Y \mathcal X}{q-1} = 1,
\qquad \quad
\frac{q  \mathcal Y \mathcal Z- \mathcal Z \mathcal Y}{q-1} = 1,
\qquad \quad
\frac{q \mathcal Z \mathcal X- \mathcal X \mathcal Z}{q-1} = 1. \label{eq:XYZ}
\end{align}
An isomorphism with the presentation in Definition  \ref{def:qGroup} 
 sends
\begin{align*}
 \mathcal X &\mapsto K^{-1} + K^{-1} E  (q-1),
\\
 \mathcal Y^{\pm 1} & \mapsto K^{\pm 1},
\\
 \mathcal Z &\mapsto K^{-1} -  F q^{-1/2}(q-1).
\end{align*}
\noindent The inverse isomorphism sends
\begin{align*}
K^{\pm 1} &\mapsto  \mathcal Y^{\pm 1}, \\
E &\mapsto  ( \mathcal Y \mathcal X-1)(q-1)^{-1},
\\
F &\mapsto  ( \mathcal Y^{-1}- \mathcal Z)q^{1/2}(q-1)^{-1}.
\end{align*}
\end{lemma}
\begin{proof} This is \cite[Lemma~5.1]{uqLUS} with $q$ replaced by $q^{1/2}$, and the parameter assignments  $\theta=-1$ and $t=0$.
\end{proof}

\begin{definition}\rm (See \cite[Definition~2.2]{equit}.) By the {\it equitable presentation} of $U_{q^{1/2}}({\mathfrak{sl}_2})$, we mean the presentation given in Lemma \ref{lem:Eq}.
We call $\mathcal X$, $\mathcal Y$, $\mathcal Y^{-1}$, $\mathcal Z$ the {\it equitable generators} for $U_{q^{1/2}}({\mathfrak{sl}_2})$.
\end{definition}

\noindent Next, we interpret \eqref{eq:UqOne}, \eqref{eq:UqTwo} in terms of the equitable generators for  $U_{q^{1/2}} (\mathfrak{sl}_2)$.
\begin{theorem} \label{lem:U2module}  On the $U_{q^{1/2}} (\mathfrak{sl}_2)$-module $V$ from Lemma \ref{lem:Umodule}, the equitable generators $\mathcal X$, $\mathcal Y$, $\mathcal Y^{-1}$, $\mathcal Z$ 
act as follows:
\begin{align*} 
\begin{tabular}[t]{c|cccc}
{\rm generator $g$ }& $\mathcal X$ & $\mathcal Y$ & $\mathcal Y^{-1}$ & $\mathcal Z$
 \\
 \hline
 {\rm action of $g$ on $V$} & $(q-1)q^{-N/2} A^-$ & $q^{N/2}A^*$ & $q^{-N/2} (A^*)^{-1}$ & $(q-1)q^{-N/2} A^+$
    \end{tabular}
\end{align*}
\end{theorem} 
\begin{proof} Compare the two rows of the table, using Definition \ref{def:Apm} and Lemmas  \ref{lem:Umodule},  \ref{lem:Eq}.
\end{proof} 

\noindent We will use the following result about $U_{q^{1/2}} (\mathfrak{sl}_2)$.

\begin{lemma} \label{lem:UqNeed} {\rm (See \cite[Proposition~7.17]{bockting}.)} Pick $0 \not=a \in \mathbb C$ such that $a^2 \not=1$.
Define 
\begin{align}
\mathcal A = a^{-1} \mathcal X + a \mathcal Z. \label{eq:MCA}
\end{align}
Then
\begin{align}
& \mathcal Y^3 \mathcal A - (\beta+1) \mathcal Y^2 \mathcal A \mathcal Y + (\beta+1) \mathcal Y \mathcal A \mathcal Y^2 - \mathcal A \mathcal Y^3 =0, \label{eq:XXX} \\
&\frac{\mathcal A^3 \mathcal Y - (\beta+1) \mathcal A^2 \mathcal Y \mathcal A + (\beta+1) \mathcal A \mathcal Y \mathcal A^2 - \mathcal Y \mathcal A^3}{(q-q^{-1})^2} = \mathcal Y \mathcal A - \mathcal A \mathcal Y, \label{eq:YYY}
\end{align}
where $\beta = q + q^{-1}$.
\end{lemma}

\noindent Recall the scalar $\varphi$ and the corresponding matrix $A$ from Definition \ref{def:weightedA}.
\begin{lemma}\label{lem:avphi}  Referring to Lemma  \ref{lem:UqNeed}, let 
\begin{align}
a = {\bf i} \varphi^{-1/2}, \qquad \qquad {\bf i}^2 =-1 \label{eq:defa}
\end{align}
\noindent  and note that $a^2 =-\varphi^{-1} \not=1$. Then on the $U_{q^{1/2}} (\mathfrak{sl}_2)$-module $V$ from Lemma \ref{lem:Umodule},
\begin{align}
\mathcal A = {\bf i} (1-q) q^{-N/2} \varphi^{-1/2} A. \label{eq:AA}
\end{align}
\end{lemma}
\begin{proof} We verify \eqref{eq:AA} as follows. Evaluate the left-hand side of \eqref{eq:AA} using  \eqref{eq:MCA}, \eqref{eq:defa} and Theorem \ref{lem:U2module}. Evaluate
the right-hand side of \eqref{eq:AA} using \eqref{eq:AAA}.
\end{proof}

\noindent Next, we obtain Propositions \ref{prop:qSerre}, \ref{thm:TDrel}  from Lemma \ref{lem:UqNeed} with $a={\bf i} \varphi^{-1/2}$.
To obtain Proposition \ref{prop:qSerre}, evaluate \eqref{eq:XXX} using $\mathcal Y=q^{N/2} A^*$ and \eqref{eq:AA}.
To obtain  Proposition \ref{thm:TDrel}, evaluate \eqref{eq:YYY} using $\mathcal Y=q^{N/2} A^*$ and \eqref{eq:AA}.

\section{$L_N(q)$ and Leonard systems}

\noindent  We continue to discuss the poset $L_N(q)$. As we computed the eigenvalues of $A$ in Proposition  \ref{prop:Weigval}, we encountered the
notion of a Leonard system. In this section, we explain in more detail what Leonard systems have to do with $L_N(q)$. Our explanation
begins with the concept of a Leonard pair.
\medskip

\noindent The concept of a Leonard pair was introduced in \cite{LS99}. Roughly speaking, a Leonard pair consists of two
multiplicity-free linear maps on a nonzero finite-dimensional vector space, that each act on an eigenbasis for the other one in an irreducible
tridiagonal fashion \cite[Definition~1.1]{LS99}. An introductory survey about Leonard pairs can be found in \cite{LSintro}.
To investigate  a  Leonard pair, it is helpful to bring in a related object called a Leonard system.
Roughly speaking, a Leonard system consists of a Leonard pair, together with appropriate orderings of their primitive idempotents \cite[Definition~1.4]{LS99}.
The Leonard systems
are classified up to isomorphism \cite[Theorem~1.9]{LS99}, \cite[Section~5]{TLT:array}. See \cite{LSnotes} for a modern treatment of the classification.
Some notable papers about Leonard pairs and systems are \cite{nomKraw, qSerre, LS24,qrac, vidunas}.
%%%In this section, we explain how Leonard pairs and systems come up in our analysis of $L_N(q)$.
\medskip

\noindent Recall the standard $T$-module $V$ for $L_N(q)$. Let $W$ denote an irreducible $T$-submodule of $V$, and let $r$ denote the endpoint of $W$.
We will show
that $A, A^*$ act on  $W$ as a Leonard pair. To do this, we show that the sequence
 $(A, \lbrace E_i \rbrace_{i=r}^{N-r}; A^*; \lbrace E_i^*\rbrace_{i=r}^{N-r})$ acts on 
$W$ as a Leonard system. 
%%%By Definition \ref{def:As} and Lemma \ref{thm:Adiag}, 
By Lemma \ref{lem:Wrep}
and Proposition \ref{prop:Weigval},
each of $A, A^*$ is multiplicity-free on $W$. For these actions, the eigenspaces  are described as follows. 
By Proposition \ref{prop:Weigval} we find that for $0 \leq i \leq N$,
$E_iW \not=0$ if and only if $r \leq i \leq N-r$, and ${\rm dim}\, E_iW=1$ for $r \leq i \leq N-r$.
 By Lemma \ref{lem:Wrep} we find that for $0 \leq i \leq N$,
$E^*_iW \not=0$ if and only if $r \leq i \leq N-r$, and ${\rm dim}\, E^*_iW=1$ for $r \leq i \leq N-r$.
By Lemma \ref{lem:triples} and since the $T$-submodule $W$ is irreducible,  the following holds on $W$:
\begin{align*}
E^*_i A E^*_j = \begin{cases} 0, & \mbox{if $\vert i-j\vert >1$}; \\
                                            \not=0, & \mbox{if $\vert i-j\vert=1$}.
                         \end{cases} \qquad \qquad (r \leq i,j\leq N-r).
\end{align*}
By   \eqref{eq:EAsE} and since the $T$-submodule $W$ is irreducible, the following holds on $W$:
\begin{align*}
E_i A^* E_j = \begin{cases} 0, & \mbox{if $\vert i-j\vert >1$}; \\
                                            \not=0, & \mbox{if $\vert i-j\vert=1$}.
                         \end{cases} \qquad \qquad (r \leq i,j\leq N-r).
\end{align*}
By these comments and \cite[Definition~1.4]{LS99},
the sequence $(A, \lbrace E_i \rbrace_{i=r}^{N-r}; A^*; \lbrace E_i^*\rbrace_{i=r}^{N-r})$ acts on 
$W$ as a Leonard system. Denote this Leonard system by $\Phi$.
 In \cite[Example~20.7]{LSnotes} there is a description of the Leonard systems that have dual $q$-Krawtchouk type. This description 
 involves six parameters that we will discuss in a moment.
 From the form of the matrices that represent $A$ in Proposition \ref{thm:Wrep}
 and $A^*$ in Lemma \ref{lem:Wrep}, and from the formula for $\lbrace \theta_i \rbrace_{i=r}^{N-r}$ in
 Definition \ref{def:theta},
 we see that the
Leonard system $\Phi$ has dual $q$-Krawtchouk type with the following parameters:
\begin{align*}
& d(\Phi)=N-2r, \qquad \qquad h(\Phi) = \frac{\varphi q^{N-r}}{q-1}, \qquad \qquad  h^*(\Phi)=q^{-r},
\\
& s(\Phi) = - \varphi^{-1} q^{2r-N-1}, 
\qquad \quad  \theta_0(\Phi) = \frac{\varphi q^{N-r}-q^r}{q-1},   \qquad \quad \theta^*_0(\Phi) = q^{-r}.
\end{align*}
\noindent See \cite{nearBip, Atten,  boyd} for more information about the Leonard systems of dual $q$-Krawtchouk type.

\section{The split decompositions of the standard module $V$}

\noindent We continue to discuss the poset $L_N(q)$. In order to motivate this section, we briefly consider a $Q$-polynomial distance-regular graph $\Gamma$ \cite{bbit, bannai,bcn,dkt}.
In \cite[Definition~5.1]{ds} we introduced the notion of a split decomposition for the standard module of $\Gamma$. For each vertex $x$ of $\Gamma$ there are
four split decompositions, see \cite[Definition~3.1]{kim1}, \cite[Definition~4.1]{kim2}, \cite[Definition~10.1]{qtet}. These four split decompositions are obtained using
the adjacency matrix of $\Gamma$ and the dual adjacency matrix of $\Gamma$ with respect to $x$. Referring to $L_N(q)$, 
in this section we use $A$ and $A^*$ to construct four analogous decompositions of the standard module $V$.
These decompositions are described in Lemma \ref{lem:Uds} and Theorem  \ref{lem:ECE} below.
\medskip

\noindent   Our construction is outlined as follows.
Recall the parameter $\varphi$ and the corresponding matrix $A$ from Definition \ref{def:weightedA}.
Recall the matrices $A^+, A^-$ from Definition \ref{def:Apm}. We will introduce a variation on $A^+, A^-$ called $A_+, A_-$. We will show that $A_+, A_-,A^*$ are related in
the same way that $A^+, A^-, A^*$ are related. We will use $A_+, A_-, A^*$ to obtain a $U_{q^{1/2}}(\mathfrak{sl}_2)$-module structure on $V$ that is slightly different from the one in Section 13.
Using the two $U_{q^{1/2}}(\mathfrak{sl}_2)$ actions on $V$, we will complete the following tasks.
 We will show that
 each of  $A^+, A^-, A_+, A_-$ is diagonalizable, with eigenvalues $q^i/(q-1)$ $(0 \leq i \leq N)$.
 We will display four bases for $V$, said to be split. We will show that each split basis diagonalizes one of $A^+, A^-, A_+, A_-$.
 We will give the action of $A, A^*$ on each split basis. We will show that on each split basis, one of $A, A^*$ acts in an upper triangular fashion
 and the other one acts in a lower triangular fashion.
 Using these actions, we will show that the eigenspace decompositions of $A^+, A^-, A_+, A_-$ satisfy  Lemma \ref{lem:Uds} and Theorem  \ref{lem:ECE}.

\begin{definition} \label{def:Apm2} Define
\begin{align*}
A_+ =  \frac{(A^*)^{-1}}{q-1}+\varphi^{-1} R, \qquad \qquad
A_- =  \frac{(A^*)^{-1}}{q-1} -\varphi (A^*)^{-1} L.
\end{align*}
\end{definition}
\begin{lemma}\label{lem:UqRel2} We have
\begin{align}
A =  \varphi A_+- A_-. \label{eq:AAA2}
\end{align}
Moreover
\begin{align}
&q A_-  A^* - A^* A_- = I, \qquad \qquad  q A^* A_+  -A_+ A^*= I, \label{eq:UqOne2}\\
&   q A_+ A_-   -A_- A_+ = \frac{q^N I}{q-1}. \label{eq:UqTwo2}
\end{align}
\end{lemma}
\begin{proof} To verify  \eqref{eq:AAA2}, eliminate $A, A_+, A_-$ using Lemma \ref{lem:Ashape} and Definition \ref{def:Apm2}.
To verify \eqref{eq:UqOne2} and \eqref{eq:UqTwo2}, eliminate $A_+, A_-$ using Definition \ref{def:Apm2} and evaluate the result using
Lemma \ref{lem:AsRL}.
\end{proof}

\begin{theorem} \label{lem:U2module2} The standard module $V$ becomes a $U_{q^{1/2}} (\mathfrak{sl}_2)$-module on which the equitable generators $\mathcal X$, $\mathcal Y$, $\mathcal Y^{-1}$, $\mathcal Z$ 
act as follows:
\begin{align*} 
\begin{tabular}[t]{c|cccc}
{\rm generator $g$ }& $\mathcal X$ & $\mathcal Y$ & $\mathcal Y^{-1}$ & $\mathcal Z$
 \\
 \hline
 {\rm action of $g$ on $V$} & $(q-1)q^{-N/2} A_-$ & $q^{N/2}A^*$ & $q^{-N/2} (A^*)^{-1}$ & $(q-1)q^{-N/2} A_+$
    \end{tabular}
\end{align*}
\end{theorem} 
\begin{proof} Compare the relations \eqref{eq:XYZ} with the relations \eqref{eq:UqOne2}, \eqref{eq:UqTwo2}.
\end{proof}

\noindent Consider the equitable generators $\mathcal X$, $\mathcal Y$,  $\mathcal Z$ of $U_{q^{1/2}} (\mathfrak{sl}_2)$. Rearranging the relations \eqref{eq:XYZ}, we obtain 
\begin{align*}
q(1-\mathcal X \mathcal Y)=1-\mathcal Y \mathcal X, \qquad \quad 
q (1-\mathcal Y \mathcal Z)= 1-\mathcal Z\mathcal Y, \qquad \quad q(1-\mathcal Z \mathcal X)=1-\mathcal  X \mathcal Z.
\end{align*}
Following \cite[Definition~5.2]{equit} we define
\begin{align*} 
  n_x &= \frac{q(1-\mathcal Y \mathcal Z)}{q-1} = \frac{1-\mathcal Z \mathcal Y}{q-1}, \\
    n_y &= \frac{q(1-\mathcal Z \mathcal X)}{q-1} = \frac{1-\mathcal X \mathcal Z}{q-1}, \\
      n_z &= \frac{q(1-\mathcal X \mathcal Y)}{q-1} = \frac{1-\mathcal Y \mathcal X}{q-1}.
\end{align*}
\noindent We are mainly interested in $n_x, n_z$.

\begin{lemma} \label{lem:nxnz} On the $U_{q^{1/2}} (\mathfrak{sl}_2)$-module $V$ from  Lemma \ref{lem:Umodule}, we have
\begin{align*}
n_x = RA^* ,   \qquad \qquad n_z =-L. 
\end{align*}
\end{lemma}
\begin{proof} To verify $n_x = RA^*$, in $n_x = (1-\mathcal Z \mathcal Y)(q-1)^{-1}$ eliminate $\mathcal Y,\mathcal Z$ using
Theorem \ref{lem:U2module}, and evaluate the result using Definition  \ref{def:Apm}. 
To verify $n_z = -L$, in $n_z = (1-\mathcal Y \mathcal X)(q-1)^{-1}$ eliminate $\mathcal X,\mathcal Y$ using
Theorem \ref{lem:U2module}, and evaluate the result using Definition  \ref{def:Apm}.
\end{proof}

\begin{lemma} \label{lem:nxnz2} On the $U_{q^{1/2}} (\mathfrak{sl}_2)$-module $V$ from Theorem \ref{lem:U2module2}, we have
\begin{align*}
n_x = -\varphi^{-1} RA^* ,   \qquad \qquad n_z =\varphi L. 
\end{align*}
\end{lemma}
\begin{proof} To verify $n_x = -\varphi^{-1} RA^*$, in $n_x = (1-\mathcal Z \mathcal Y)(q-1)^{-1}$ eliminate $\mathcal Y,\mathcal Z$ using
Theorem \ref{lem:U2module2}, and evaluate the result using Definition  \ref{def:Apm2}. 
To verify $n_z = \varphi L$, in $n_z = (1-\mathcal Y \mathcal X)(q-1)^{-1}$ eliminate $\mathcal X,\mathcal Y$ using
Theorem \ref{lem:U2module2}, and evaluate the result using Definition  \ref{def:Apm2}.
\end{proof}

\noindent We recall the $q$-exponential function; see for example \cite[p.~204]{tanisaki}. 
 Let $\psi$ denote a linear operator that acts on finite-dimensional $U_{q^{1/2}} (\mathfrak{sl}_2)$-modules
in a nilpotent fashion. We define
\begin{align}
{\rm exp}_{q^{1/2}}(\psi) = \sum_{i=0}^\infty \frac{ q^{i(i-1)/2} }{ \lbrack i \rbrack^{!}_q} \psi^i. \label{eq:Exp1}
\end{align}
We view ${\rm exp}_{q^{1/2}}(\psi)$ as a linear operator that acts on finite-dimensional $U_{q^{1/2}} (\mathfrak{sl}_2)$-modules.
On such a module  ${\rm exp}_{q^{1/2}}(\psi)$ is invertible; the inverse is 
\begin{align}
{\rm exp}_{q^{-1/2}}(-\psi) = \sum_{i=0}^\infty \frac{ (-1)^i }{ \lbrack i \rbrack^{!}_q} \psi^i. \label{eq:Exp2}
\end{align}

\begin{lemma} {\rm (See \cite[Sections~5,~6]{equit}.)} \label{eq:expE}The following {\rm (i)--(iii)} hold on each finite-dimensional $U_{q^{1/2}} (\mathfrak{sl}_2)$-module:
\begin{enumerate}
\item[\rm (i)] $\mathcal X\,{\rm exp}_{q^{1/2}} (n_x) = {\rm exp}_{q^{1/2}} (n_x) \, (\mathcal X +\mathcal Y-\mathcal Y^{-1})$;
\item[\rm (ii)] $\mathcal Y\,{\rm exp}_{q^{1/2}} (n_x) = {\rm exp}_{q^{1/2}} (n_x) \, \mathcal Y \mathcal Z \mathcal Y$;
\item[\rm (iii)] $\mathcal Z\,{\rm exp}_{q^{1/2}} (n_x) = {\rm exp}_{q^{1/2}} (n_x) \, \mathcal Y^{-1}$.
\end{enumerate}
\end{lemma}

\begin{lemma} {\rm (See \cite[Sections~5,~6]{equit}.)} \label{eq:expE2}The following {\rm (i)--(iii)} hold on each finite-dimensional $U_{q^{1/2}} (\mathfrak{sl}_2)$-module:
\begin{enumerate}
\item[\rm (i)] ${\rm exp}_{q^{1/2}} (n_z)\, \mathcal X= \mathcal Y^{-1} \,{\rm exp}_{q^{1/2}} (n_z) $;
\item[\rm (ii)] ${\rm exp}_{q^{1/2}} (n_z)\, \mathcal Y= \mathcal Y \mathcal X \mathcal Y\,  {\rm exp}_{q^{1/2}} (n_z) $;
\item[\rm (iii)] ${\rm exp}_{q^{1/2}} (n_z) \,\mathcal Z= (\mathcal Z + \mathcal Y - \mathcal Y^{-1})\, {\rm exp}_{q^{1/2}} (n_z) $.
\end{enumerate}
\end{lemma}

\begin{definition}\label{def:ydd} \rm For $y \in X$ define
\begin{align*}
&y^{\downarrow \downarrow} = {\rm exp}_{q^{-1/2}} (L) {\hat y}, \qquad \qquad  y^{\uparrow \downarrow} = {\rm exp}_{q^{1/2}} (RA^*) {\hat y}, \\
&y^{\downarrow \uparrow} = {\rm exp}_{q^{-1/2}} (-\varphi L) {\hat y}, \qquad \qquad  y^{\uparrow \uparrow} = {\rm exp}_{q^{1/2}} (-\varphi^{-1}RA^*) {\hat y}.
\end{align*}
\end{definition}

\begin{lemma} \label{lem:yymeaning} For $y \in X$ we have
\begin{align*}
&y^{\downarrow \downarrow} = \sum_{z \leq y} {\hat z}, \qquad \qquad y^{\uparrow \downarrow} = \sum_{y \leq z } {\hat z} q^{ -{\rm dim}\,y ( {\rm dim}\,z-{\rm dim}\,y)  }, \\
&y^{\downarrow \uparrow} = \sum_{z \leq y} {\hat z}(-\varphi)^{{\rm dim}\,y-{\rm dim}\,z}, \qquad \quad y^{\uparrow \uparrow} = \sum_{y \leq z } {\hat z} q^{ -{\rm dim}\,y ( {\rm dim}\,z-{\rm dim}\,y)  }  (-\varphi)^{{\rm dim}\,y-{\rm dim}\,z}.
\end{align*}
\end{lemma}
\begin{proof} Use Definition \ref{def:ydd} together with
 \eqref{eq:Exp1} (for $y^{\uparrow \downarrow}, y^{\uparrow \uparrow}$) or 
  \eqref{eq:Exp2} (for $y^{\downarrow \downarrow}, y^{\downarrow \uparrow}$).
\end{proof}

\begin{lemma}\label{lem:4basis} Each of the following is a basis for the vector space $V$:
\begin{align*}
\lbrace y^{\downarrow \downarrow} \rbrace_{y \in X}, \qquad \quad
\lbrace y^{\uparrow \downarrow} \rbrace_{y \in X}, \qquad \quad
\lbrace y^{\downarrow \uparrow} \rbrace_{y \in X}, \qquad \quad
\lbrace y^{\uparrow \uparrow} \rbrace_{y \in X}.
\end{align*}
\end{lemma}
\begin{proof} Consider the vectors $\lbrace y^{\downarrow \downarrow} \rbrace_{y \in X}$. List the elements of $X$ in order  of nondecreasing dimension.
 For $y \in X$ write $y^{\downarrow \downarrow}$ in the basis $\lbrace {\hat x}\vert x \in X\rbrace$. The coefficient matrix is upper triangular and has all diagonal entries 1.
 This matrix is invertible, so the vectors $\lbrace y^{\downarrow \downarrow} \rbrace_{y \in X}$ form a basis for $V$. For the other three bases, the proof is similar.
\end{proof}

\begin{lemma} \label{lem:Actionydd} For $y \in X$ we have
\begin{align*}
A^* y^{\downarrow \downarrow} &=   q^{-{\rm dim}\,y}  y^{\downarrow \downarrow} + (q-1)q^{-{\rm dim}\,y} \sum_{y \,{\rm covers}\,z} z^{\downarrow \downarrow}, \\
A^+ y^{\downarrow \downarrow} &= \frac{q^{N-{\rm dim}\,y}}{q-1} y^{\downarrow \downarrow} - \sum_{z \,{\rm covers}\,y} z^{\downarrow \downarrow},\\
A^- y^{\downarrow \downarrow} & = \frac{q^{{\rm dim}\,y}}{q-1} y^{\downarrow \downarrow},\\
A y^{\downarrow \downarrow} &= \frac{\varphi q^{{\rm dim}\,y}-q^{N-{\rm dim}\,y}}{q-1} y^{\downarrow \downarrow} + \sum_{z \,{\rm covers}\,y} z^{\downarrow \downarrow} .
\end{align*}
\end{lemma}
\begin{proof}  To obtain the first three equations,
consider the $U_{q^{1/2}} (\mathfrak{sl}_2)$-module $V$ from Lemma  \ref{lem:Umodule} and Theorem \ref{lem:U2module}. We adjust each equation in Lemma \ref{eq:expE2}, by multiplying each 
side on the left and right by the inverse of ${\rm exp}_{q^{1/2}}(n_z)$.  In the resulting equation, apply each side to $\hat y$, and evaluate the result using
Definitions  \ref{def:Apm}, \ref{def:ydd} and \eqref{eq:AsAct} along with  Lemmas  \ref{lem:RLact},  \ref{lem:nxnz} and Theorem  \ref{lem:U2module}. This yields the first three equations.
 To get the fourth equation, use  \eqref{eq:AAA}.
\end{proof}

\begin{lemma} \label{lem:Actionyud} For $y \in X$ we have
\begin{align*}
A^* y^{\uparrow \downarrow} &=   q^{-{\rm dim}\,y}  y^{\uparrow \downarrow} - (q-1)q^{-2\, {\rm dim}\,y-1} \sum_{z \,{\rm covers}\,y} z^{\uparrow \downarrow}, \\
A^+ y^{\uparrow \downarrow} & = \frac{q^{{\rm dim}\,y}}{q-1} y^{\uparrow \downarrow},\\
A^- y^{\uparrow \downarrow} &= \frac{q^{N-{\rm dim}\,y}}{q-1} y^{\uparrow \downarrow} +q^{{\rm dim}\,y-1} \sum_{y \,{\rm covers}\,z} z^{\uparrow \downarrow},\\
A y^{\uparrow \downarrow} &= \frac{\varphi q^{N-{\rm dim}\,y}-q^{{\rm dim}\,y}}{q-1} y^{\uparrow \downarrow} + \varphi q^{{\rm dim}\,y-1}\sum_{y \,{\rm covers}\,z} z^{\uparrow \downarrow} .
\end{align*}
\end{lemma}
\begin{proof}
To obtain the first three equations,
consider the $U_{q^{1/2}} (\mathfrak{sl}_2)$-module $V$ from Lemma  \ref{lem:Umodule} and Theorem \ref{lem:U2module}. For each equation in Lemma \ref{eq:expE},  apply each side to $\hat y$, and evaluate the result using
Definitions  \ref{def:Apm}, \ref{def:ydd} and \eqref{eq:AsAct} along with  Lemmas  \ref{lem:RLact},  \ref{lem:nxnz} and Theorem  \ref{lem:U2module}. This yields the first three equations.
 To get the fourth equation, use  \eqref{eq:AAA}.
\end{proof}

\begin{lemma} \label{lem:Actionydu} For $y \in X$ we have
\begin{align*}
A^* y^{\downarrow \uparrow} &=   q^{-{\rm dim}\,y}  y^{\downarrow \uparrow} - (q-1) \varphi q^{-{\rm dim}\,y} \sum_{y \,{\rm covers}\,z} z^{\downarrow \uparrow}, \\
A_+ y^{\downarrow \uparrow} &= \frac{q^{N-{\rm dim}\,y}}{q-1} y^{\downarrow \uparrow} +\varphi^{-1} \sum_{z \,{\rm covers}\,y} z^{\downarrow \uparrow},\\
A_- y^{\downarrow \uparrow} & = \frac{q^{{\rm dim}\,y}}{q-1} y^{\downarrow \uparrow},\\
A y^{\downarrow \uparrow} &= \frac{\varphi q^{N-{\rm dim}\,y}-q^{{\rm dim}\,y}}{q-1} y^{\downarrow \uparrow} + \sum_{z \,{\rm covers}\,y} z^{\downarrow \uparrow} .
\end{align*}
\end{lemma}
\begin{proof} Similar to the proof of Lemma \ref{lem:Actionydd}.
To obtain the first three equations,
consider the $U_{q^{1/2}} (\mathfrak{sl}_2)$-module $V$ from Theorem \ref{lem:U2module2}. We adjust each equation in Lemma \ref{eq:expE2}, by multiplying each 
side on the left and right by the inverse of ${\rm exp}_{q^{1/2}}(n_z)$.  In the resulting equation, apply each side to $\hat y$, and evaluate the result using
Definitions \ref{def:Apm2}, \ref{def:ydd}
and \eqref{eq:AsAct}  along with  Lemmas  \ref{lem:RLact},  \ref{lem:nxnz2} and Theorem  \ref{lem:U2module2}. This yields the first three equations.
 To get the fourth equation, use  \eqref{eq:AAA2}.
\end{proof}

\begin{lemma} \label{lem:Actionyuu} For $y \in X$ we have
\begin{align*}
A^* y^{\uparrow \uparrow} &=   q^{-{\rm dim}\,y}  y^{\uparrow \uparrow} + (q-1) \varphi^{-1} q^{-2\,{\rm dim}\,y-1} \sum_{z \,{\rm covers}\,y} z^{\uparrow \uparrow}, \\
A_+ y^{\uparrow \uparrow} & = \frac{q^{{\rm dim}\,y}}{q-1} y^{\uparrow \uparrow},\\
A_- y^{\uparrow \uparrow} &= \frac{q^{N-{\rm dim}\,y}}{q-1} y^{\uparrow \uparrow} -\varphi q^{{\rm dim}\,y-1} \sum_{y \,{\rm covers}\,z} z^{\uparrow \uparrow},\\
A y^{\uparrow \uparrow} &= \frac{\varphi q^{{\rm dim}\,y}-q^{N-{\rm dim}\,y}}{q-1} y^{\uparrow \uparrow} +\varphi q^{{\rm dim}\,y-1} \sum_{y \,{\rm covers}\,z} z^{\uparrow \uparrow} .
\end{align*}
\end{lemma}
\begin{proof} Similar to the proof of Lemma \ref{lem:Actionyud}.
To obtain the first three equations,
consider the $U_{q^{1/2}} (\mathfrak{sl}_2)$-module $V$ from Theorem \ref{lem:U2module2}. For each equation in Lemma \ref{eq:expE}, apply each side to $\hat y$, and evaluate the result using
Definitions \ref{def:Apm2}, \ref{def:ydd}
 and \eqref{eq:AsAct}  along with Lemmas  \ref{lem:RLact}, \ref{lem:nxnz2} and Theorem  \ref{lem:U2module2}. This yields the first three equations.
 To get the fourth equation, use  \eqref{eq:AAA2}.
\end{proof}

\begin{lemma}\label{lem:AAAAeigval}
Each of $A^-, A^+, A_-, A_+$ is diagonalizable with eigenvalues 
\begin{align*}
\frac{q^i}{q-1} \qquad \qquad (0 \leq i \leq N).
\end{align*}
\noindent For each matrix and $0 \leq i \leq N$, the eigenspace with eigenvalue $q^i/(q-1)$ has dimension $\binom{N}{i}_q$.
\end{lemma}
\begin{proof} For $A^-$, the result follows from Lemmas \ref{lem:size}, \ref{lem:4basis} and the $A^-$ action described in Lemma \ref{lem:Actionydd}.
For $A^+, A_-, A_+$ the proof is similar.
\end{proof} 

\begin{definition}\label{def:Ui} \rm For $A^-, A^+, A_-, A_+$ we name their eigenspaces as follows. For $0 \leq i \leq N$,
\begin{align*} 
\begin{tabular}[t]{c|cccc}
{\rm matrix }& $A^-$ & $A^+$ & $A_-$ & $A_+$
 \\
 \hline
 {\rm eigenspace for $q^i/(q-1)$} & $U_i^{\downarrow \downarrow} $ & $U_i^{\uparrow \downarrow}$ & $ U_i^{\downarrow \uparrow}$ & $U_i^{\uparrow \uparrow}$
    \end{tabular}
\end{align*}
\noindent For notational convenience, we define 
\begin{align*}
 U_{j}^{\downarrow \downarrow}=0, \qquad U_{j}^{\uparrow \downarrow}=0,  \qquad U_{j}^{\downarrow \uparrow}=0,  \qquad U_{j}^{\uparrow \uparrow}=0
\end{align*}
\noindent for $j=-1$ and $j=N+1$.
\end{definition}

\begin{lemma} \label{lem:Uds} The following sums are direct:
\begin{align*}
&V = \sum_{i=0}^N U_i^{\downarrow \downarrow}, \qquad \quad V = \sum_{i=0}^N U_i^{\uparrow \downarrow}, 
\qquad \quad V = \sum_{i=0}^N U_i^{\downarrow \uparrow}, \qquad \quad V = \sum_{i=0}^N U_i^{\uparrow \uparrow}.
\end{align*}
\end{lemma}
\begin{proof} By Lemma \ref{lem:AAAAeigval}, Definition \ref{def:Ui}, and linear algebra.
\end{proof}

\begin{lemma} \label{lem:Ubasis} For $0 \leq i \leq N$, we give a basis for each subspace $U_i^{\downarrow \downarrow}, U_i^{\uparrow \downarrow}, U_i^{\downarrow \uparrow}, U_i^{\uparrow \uparrow}$ from
Definition \ref{def:Ui}:
\begin{align*} 
\begin{tabular}[t]{c|c}
{\rm subspace }& {\rm basis} \\
\hline
 $U_i^{\downarrow \downarrow}$ & $\lbrace y^{\downarrow \downarrow} \vert y \in X, \;{\rm dim}\,y=i\rbrace$ \\
$U_i^{\uparrow \downarrow}$ & $\lbrace y^{\uparrow \downarrow} \vert y \in X, \;{\rm dim}\,y=i\rbrace$ \\
$U_i^{\downarrow \uparrow}$ & $\lbrace y^{\downarrow \uparrow} \vert y \in X, \;{\rm dim}\,y=i\rbrace$ \\
$U_i^{\uparrow \uparrow}$ & $\lbrace y^{\uparrow \uparrow} \vert y \in X, \;{\rm dim}\,y=i\rbrace$ 
     \end{tabular}
\end{align*}
\end{lemma}
\begin{proof} By Lemmas \ref{lem:Actionydd}--\ref{lem:Actionyuu}
and
Definition \ref{def:Ui}.
\end{proof} 

\begin{lemma} \label{lem:AAAAsplit} For $0 \leq i \leq N$ we have
\begin{align*}
&(A- \theta_{N-i} I) U_i^{\downarrow \downarrow} \subseteq U_{i+1}^{\downarrow \downarrow}, \qquad \qquad (A^*- \theta^*_{i} I) U_i^{\downarrow \downarrow} \subseteq U_{i-1}^{\downarrow \downarrow},\\
&(A- \theta_{i} I) U_i^{\uparrow \downarrow} \subseteq U_{i-1}^{\uparrow \downarrow}, \qquad \qquad (A^*- \theta^*_{i} I) U_i^{\uparrow \downarrow} \subseteq U_{i+1}^{\uparrow \downarrow},\\
&(A- \theta_{i} I) U_i^{\downarrow \uparrow} \subseteq U_{i+1}^{\downarrow \uparrow}, \qquad \qquad (A^*- \theta^*_{i} I) U_i^{\downarrow \uparrow} \subseteq U_{i-1}^{\downarrow \uparrow},\\
&(A- \theta_{N-i} I) U_i^{\uparrow \uparrow} \subseteq U_{i-1}^{\uparrow \uparrow}, \qquad \qquad (A^*- \theta^*_{i} I) U_i^{\uparrow \uparrow} \subseteq U_{i+1}^{\uparrow \uparrow}.
\end{align*}
\end{lemma}
\begin{proof}  By Lemmas \ref{lem:Actionydd}--\ref{lem:Actionyuu}
and 
Lemma \ref{lem:Ubasis}.
\end{proof}

\begin{lemma} \label{lem:3sum} The following {\rm (i)--(iv)} hold for $0 \leq i \leq N$.
\begin{enumerate}
\item[\rm (i)] These sums are equal:
\begin{align*}
E^*_0V + E^*_1V+\cdots + E^*_iV, \quad U_0^{\downarrow \downarrow} +U_1^{\downarrow \downarrow}+\cdots+U_i^{\downarrow \downarrow}, \quad
 U_0^{\downarrow \uparrow} +U_1^{\downarrow \uparrow}+\cdots+U_i^{\downarrow \uparrow}.
 \end{align*}
\item[\rm (ii)] These sums are equal:
\begin{align*}
E^*_NV + E^*_{N-1}V+\cdots + E^*_iV, \quad U_N^{\uparrow \downarrow} +U_{N-1}^{\uparrow \downarrow}+\cdots+U_i^{\uparrow \downarrow}, \quad
 U_N^{\uparrow \uparrow} +U_{N-1}^{\uparrow \uparrow}+\cdots+U_i^{\uparrow \uparrow}.
 \end{align*}
\item[\rm (iii)] These sums are equal:
\begin{align*}
E_0V + E_1V+\cdots + E_iV, \quad U_N^{\downarrow \downarrow} +U_{N-1}^{\downarrow \downarrow}+\cdots+U_{N-i}^{\downarrow \downarrow}, \quad
 U_0^{\uparrow \downarrow } +U_1^{\uparrow \downarrow }+\cdots+U_i^{\uparrow \downarrow}.
 \end{align*}
\item[\rm (iv)] These sums are equal:
\begin{align*}
E_NV + E_{N-1}V+\cdots + E_iV, \quad U_N^{\downarrow \uparrow} +U_{N-1}^{\downarrow \uparrow}+\cdots+U_i^{\downarrow \uparrow}, \quad
 U_0^{\uparrow \uparrow} +U_1^{\uparrow \uparrow}+\cdots+U_{N-i}^{\uparrow \uparrow}.
 \end{align*}
\end{enumerate}
\end{lemma}
\begin{proof} (i) The three sums have the same dimension, by  \eqref{eq:ctwo} and Lemmas \ref{lem:AAAAeigval}, \ref{lem:Uds} along with Definition \ref{def:Ui}. The last two sums are contained in the first sum, by Lemmas  \ref{lem:yymeaning}, \ref{lem:Ubasis}.
By these comments, the three sums are equal. \\
\noindent (ii) Similar to the proof of (i). \\
\noindent (iii) The three sums have the same dimension, by Lemmas \ref{lem:Com2}, \ref{lem:AAAAeigval}, \ref{lem:Uds} and Definition \ref{def:Ui}. We show that the last two sums are contained in the first sum.
Let $S_i^{\downarrow \downarrow}$ (resp. $S_i^{\uparrow \downarrow}$) denote the second sum (resp. third sum). Define 
\begin{align*}
P_i = (A-\theta_0 I) (A - \theta_1 I ) \cdots (A-\theta_i I).
\end{align*}
\noindent By linear algebra,
\begin{align*}
E_0V+ E_1V+\cdots + E_iV = \lbrace v \in V \vert P_iv=0\rbrace.
\end{align*}
By Lemma \ref{lem:AAAAsplit} we obtain $P_i S_i^{\downarrow \downarrow} =0$ and $P_i S_i^{\uparrow \downarrow} =0$. By these comments,
\begin{align*}
S_i^{\downarrow \downarrow}, S_i^{\uparrow \downarrow} \subseteq E_0V+E_1V+\cdots+E_iV.
\end{align*}
We have shown that the last two sums are contained in the first sum.
It follows that the three sums are equal. \\
\noindent (iv) Similar to the proof of (iii).
\end{proof}

\begin{theorem} \label{lem:ECE} For $0 \leq i \leq N$ we have
\begin{align*}
U_i^{\downarrow \downarrow} &= (E^*_0V+E^*_1V+\cdots + E^*_iV) \cap (E_0V+E_1V+\cdots+E_{N-i}V), \\
U_i^{\uparrow \downarrow} &= (E^*_NV+E^*_{N-1}V+\cdots + E^*_iV) \cap (E_0V+E_1V+\cdots+E_{i}V), \\
U_i^{\downarrow \uparrow} &= (E^*_0V+E^*_1V+\cdots + E^*_iV) \cap (E_NV+E_{N-1}V+\cdots+E_{i}V), \\
U_i^{\uparrow \uparrow} &= (E^*_NV+E^*_{N-1}V+\cdots + E^*_iV) \cap (E_NV+E_{N-1}V+\cdots+E_{N-i}V).
\end{align*}
\end{theorem}
\begin{proof} We verify the first equation. By Lemma  \ref{lem:Uds} and Lemma \ref{lem:3sum}(i),(iii) we have
\begin{align*} 
U_i^{\downarrow \downarrow} &= (U_0^{\downarrow \downarrow}+U_1^{\downarrow \downarrow}+\cdots + U_i^{\downarrow \downarrow}) \cap 
(U_i^{\downarrow \downarrow}+U_{i+1}^{\downarrow \downarrow}+\cdots + U_N^{\downarrow \downarrow}) \\
&=
(E^*_0V+E^*_1V+\cdots + E^*_iV) \cap (E_0V+E_1V+\cdots + E_{N-i}V).
\end{align*}
The other three equations are similarly verified.
\end{proof}

\noindent The equations in Theorem \ref{lem:ECE} resemble  \cite[Definition~3.1]{kim1} and \cite[Definition~5.1]{ds}. For this reason, we call the sums in Lemma \ref{lem:Uds}
the {\it split decompositions of $V$} with respect to $A$ and the vertex $\bf 0$. We also call the four bases in Lemma \ref{lem:4basis} the {\it split bases for $V$} with respect to $A$ and the vertex $\bf 0$.
\medskip

\section{Acknowledgement} 
\noindent The author thanks Roghayeh Maleki and Kazumasa Nomura for reading the manuscript carefully and offering helpful comments.
% Bill Martin for many helpful discussions.  The author thanks Kazumasa Nomura for reading the manuscript carefully, and sending comments.
%The author thanks Hajime Tanaka for  pointing out reference
% \cite{tanakaCode} in connection with Lemma  \ref{lem:orb}, and reference \cite{tanakaDual} in connection with  Conjecture \ref{conj:qhij}.

 %%%%%%%%%%%%%%%%%%%%%%%%%%%%%%%%%%%%%%%%%%%%%%%%%%%%%%%%%%%%%%
 %%%%%%%%%%%%%%%%%%%%%%%%%%%%%%%%%%%%%%%%%%%%%%%%%%%%%%%%%%%%%%
 %%%%%%%%%%%%%%%%%%%%%%%%%%%%%%%%%%%%%%%%%%%%%%%%%%%%%%%%%%%%%%
 %%%%%%%%%%%%%%%%%%%%%%%%%%%%%%%%%%%%%%%%%%%%%%%%%%%%%%%%%%%%%%
 %%%%%%%%%%%%%%%%%%%%%%%%%%%%%%%%%%%%%%%%%%%%%%%%%%%%%%%%%%%%%%
 %%%%%%%%%%%%%%%%%%%%%%%%%%%%%%%%%%%%%%%%%%%%%%%%%%%%%%%%%%%%%%
 %%%%%%%%%%%%%%%%%%%%%%%%%%%%%%%%%%%%%%%%%%%%%%%%%%%%%%%%%%%%%%
 %%%%%%%%%%%%%%%%%%%%%%%%%%%%%%%%%%%%%%%%%%%%%%%%%%%%%%%%%%%%%%

\bigskip

%\noindent Kazumasa Nomura \hfil\break
%\noindent Tokyo Medical and Dental University \hfil\break
%\noindent Kohnodai Ichikawa 272-0827 Japan \hfil\break
%\noindent email: {\tt knomura@pop11.odn.ne.jp} \hfil\break

\noindent Paul Terwilliger \hfil\break
\noindent Department of Mathematics \hfil\break
\noindent University of Wisconsin \hfil\break
\noindent 480 Lincoln Drive \hfil\break
\noindent Madison, WI 53706-1388 USA \hfil\break
\noindent email: {\tt terwilli@math.wisc.edu }\hfil\break

\section{Statements and Declarations}

\noindent {\bf Funding}: The author declares that no funds, grants, or other support were received during the preparation of this manuscript.
\medskip

\noindent  {\bf Competing interests}:  The author  has no relevant financial or non-financial interests to disclose.
\medskip

\noindent {\bf Data availability}: All data generated or analyzed during this study are included in this published article.

\end{document}